\documentclass[11pt,reqno]{amsart}

\usepackage{amsmath,amsthm,verbatim,amssymb,amsfonts,amscd, graphicx}
\usepackage{graphics}
\usepackage{geometry}
\theoremstyle{plain}

\newtheorem{conjecture}{Conjecture}

\newtheorem{defn}{Definition}[section]
\newtheorem{thm}[defn]{Theorem}
\newtheorem{cor}[defn]{Corollary}
\newtheorem{nota}[defn]{Notation}
\newtheorem{convention}[defn]{Convention}
\newtheorem{prop}[defn]{Proposition}
\newtheorem{remark}[defn]{Remark}
\newtheorem{lm}[defn]{Lemma}
\newtheorem{fact}[defn]{Fact}
\newtheorem{construction}[defn]{Construction}

\newtheorem{case}{Case}
\newtheorem{condition}{Condition}

\newtheorem{case2}{Case}
\newtheorem{assume}[defn]{Assumption}

\newtheorem*{1-mfd}{Proposition \ref{1-manifold}}

\newtheorem*{Dehn filling}{Proposition \ref{Dehn filling}}
\usepackage{extarrows}
\usepackage{mathrsfs}
\usepackage{galois}
\usepackage[colorlinks,citecolor = red, linkcolor=blue,hyperindex,CJKbookmarks]{hyperref}
\usepackage[all]{hypcap}
\usepackage[all]{xy}
\usepackage{tikz-cd}
\usepackage{latexsym}
\usepackage{subfigure}

\newcommand\blfootnote[1]{%
	\begingroup
	\renewcommand\thefootnote{}\footnote{#1}%
	\addtocounter{footnote}{-1}%
	\endgroup
}

\begin{document}
	
\title{Left orderability and taut foliations with orderable cataclysm}
\author{Bojun Zhao}
\address{Department of Mathematics, University at Buffalo, Buffalo, NY 14260, USA}
\email{bojunzha@buffalo.edu}
\maketitle

\begin{abstract}
	Let $M$ be a connected, closed, orientable, irreducible $3$-manifold.
	We show that:
	if $M$ admits a co-orientable taut foliation $\mathcal{F}$ with orderable cataclysm,
	then $\pi_1(M)$ is left orderable.
	This provides an elementary proof that 
	$\pi_1(M)$ is left orderable
	if $M$ admits an Anosov flow with a co-orientable stable foliation
	without using Thurston's universal circle action.
	Furthermore,
	for every closed orientable 3-manifold that admits a pseudo-Anosov flow $X$ with a co-orientable stable foliation,
	our result applies to infinitely many of Dehn fillings along 
	the union of singular orbits of $X$.
\end{abstract}


\section{Introduction}

Throughout this paper,
$M$ will always be a connected, closed, orientable, irreducible $3$-manifold,
and $\widetilde{M}$ will always be the universal cover of $M$.
The \emph{stable/unstable foliation} of an Anosov flow will always mean
its weak stable/unstable foliation.

A motivation of our work is 
the L-space conjecture,
which is proposed by Boyer-Gordon-Watson 
in \hyperref[BGW]{[BGW]} and 
by Juh\'asz in \hyperref[J]{[J]}:

\begin{conjecture}[L-space conjecture]\label{L-space conjecture}
	The following statements are equivalent for $M$:
	
	(1)
	$M$ is a non-L-space.
	
	(2)
	$\pi_1(M)$ is left orderable.
	
	(3)
	$M$ admits a co-orientable taut foliation $\mathcal{F}$.
\end{conjecture}

The implication (3) $\Rightarrow$ (1) is
proved by Ozsv\'ath and Szab\'o in \hyperref[OS]{[OS]}
(see also \hyperref[Bo]{[Bo]}, \hyperref[KR]{[KR]}).
Gabai (\hyperref[G1]{[G1]}) proves that $M$ admits co-orientable taut foliations if
$M$ has positive first Betti number.
Boyer, Rolfsen and Wiest (\hyperref[BRW]{[BRW]}) prove that
$\pi_1(M)$ is left orderable if $b_1(M) > 0$.
Conjecture \ref{L-space conjecture} is proved when $M$ is a graph manifold by
a combined effort of
\hyperref[BC]{[BC]}, \hyperref[R]{[R]}, \hyperref[HRRW]{[HRRW]}.
The implication (3) $\Rightarrow$ (2) is proved when $\mathcal{F}$ 
has one-sided branching in \hyperref[Zh]{[Zh]}.

Let $\mathcal{F}$ be a taut foliation of $M$,
and let $\widetilde{\mathcal{F}}$ be the pull-back foliation of $\mathcal{F}$ in $\widetilde{M}$.
A \emph{cataclysm} of $\widetilde{\mathcal{F}}$ is
a maximal collection of points in the leaf space of $\widetilde{\mathcal{F}}$ 
which are pairwise non-separated (see Definition \ref{cataclysm}, Definition \ref{cataclysm for foliations}).
We say that $\mathcal{F}$ has \emph{orderable cataclysm} if
there is a linear order at every cataclysm of $\widetilde{\mathcal{F}}$ which
is invariant under the deck transformations of $\widetilde{M}$.
Our main theorem is

\begin{thm}\label{orderable cataclysm}
	If $M$ admits a co-orientable taut foliation $\mathcal{F}$ with orderable cataclysm,
	then $\pi_1(M)$ is left orderable.
	In particular,
	if $\mathcal{F}$ has orderable cataclysm and has two-sided branching,
	then the set of ends of the leaf space of $\widetilde{\mathcal{F}}$ 
	(i.e. the pull-back foliation of $\mathcal{F}$ in $\widetilde{M}$)
	has a linear order invariant under
	the deck transformations of $\widetilde{M}$.
\end{thm}

We provide some applications of Theorem \ref{orderable cataclysm} as follows.

\subsection{Examples: Anosov flows}\label{subsection 1.1}

Fenley (\hyperref[F3]{[F3]}) proves that
the stable and unstable foliations of Anosov flows have orderable cataclysm,
so Theorem \ref{orderable cataclysm} applies to them.

\begin{cor}\label{Anosov}
	If $M$ admits a co-orientable taut foliation which is
	the stable foliation of an Anosov flow,
	then $\pi_1(M)$ is left orderable.
	In particular,
	if $M$ admits an Anosov flow,
	then either $\pi_1(M)$ is left orderable or 
	$\pi_1(M)$ has an index $2$ left orderable subgroup.
\end{cor}

\begin{remark}\rm
Corollary \ref{Anosov} also follows directly from
Thurston's universal circle action (\hyperref[T]{[T]}, \hyperref[CD]{[CD]}).
We briefly describe its background below.
Given a taut foliation $\mathcal{F}$ of $M$ with 
leaves having a continuously varying leafwise hyperbolic metric,
Thurston constructs a nontrivial homomorphism 
$$\rho_{\text{univ}}: \pi_1(M) \to Homeo(S^{1})$$
associated to $\mathcal{F}$,
called the \emph{universal circle action}.
The universcal circle action has the following properties:
(1)
when $\mathcal{F}$ is co-orientable,
$\rho_{\text{univ}}(\pi_1(M)) \subseteq Homeo_+(S^{1})$,
(2)
when $M$ is atoroidal,
$\rho_{\text{univ}}$ is injective.
Moreover,
$\rho_{\text{univ}}$ lifts to an action of $\pi_1(M)$ on $\mathbb{R}$ if
the tangent plane field of $\mathcal{F}$ has zero Euler class
(see, for example, \hyperref[BH]{[BH, Section 7]}).

Suppose that $\mathcal{F}$ is the stable foliation of an Anosov flow $X$ in $M$ and
that $\mathcal{F}$ is co-orientable. 
We denote by $T\mathcal{F}$ the tangent plane field of $\mathcal{F}$ and
denote by $e(T\mathcal{F})$ the Euler class of $T\mathcal{F}$.
Since $X$ is a section of $T\mathcal{F}$,
$e(T\mathcal{F}) = 0$.
If $b_1(M) = 0$,
then the leaves of $\mathcal{F}$ have a continuous varying leafwise hyperbolic metric
(\hyperref[BH]{[BH, Theorem 8.3]}),
and thus $\mathcal{F}$ has 
an associated universal circle action that lifts to a nontrivial action of $\pi_1(M)$ on $\mathbb{R}$.
By \hyperref[BRW]{[BRW, Theorem 3.2]},
$\pi_1(M)$ is left orderable.
If $b_1(M) > 0$,
then we know from \hyperref[BRW]{[BRW, Lemma 3.3]} that $\pi_1(M)$ is left orderable.
Thus,
the universal circle action implies that 
$\pi_1(M)$ is left orderable.
\end{remark}

The universal circle action
involves much subtle geometric and dynamical information about the foliations,
but the proof of Theorem \ref{orderable cataclysm} 
doesn't rely on this information.
Theorem \ref{orderable cataclysm} provides 
an elementary explanation for the left orderability of $\pi_1(M)$ that
is only related to the leaf space:
if $M$ admits an Anosov flow with
a co-orientable, non-$\mathbb{R}$-covered stable foliation $\mathcal{F}$
(which has two-sided branching, see Remark \ref{non-R-covered}),
then we can obtain a left-invariant order of $\pi_1(M)$ from
the set of ends of the corresponding leaf space
(see Subsection \ref{subsection 3.3}).

\begin{remark}\rm\label{non-R-covered}
	We provide some resources about Anosov flows with two-sided branching stable and unstable foliations.
	As shown in \hyperref[F2]{[F2]},
	stable and unstable foliations of any Anosov flows are either $\mathbb{R}$-covered or
	have two-sided branching.
	There are many examples of Anosov flows with two-sided branching stable and unstable foliations,
	e.g. \hyperref[BL]{[BL]}, \hyperref[Ba1]{[Ba1]}.
	See more information about such Anosov flows in \hyperref[F1]{[F1]}, \hyperref[F3]{[F3]}.
\end{remark}

\subsection{Examples: Dehn fillings and pseudo-Anosov flows with 
	co-orientable stable foliations}

Suppose $M$ admits a pseudo-Anosov flow $X$ with
a co-orientable stable foliation $\mathcal{F}^{s}$.
Let $\{\gamma_1, \ldots, \gamma_n\}$ be the set of singular orbits of $X$ and
let $\gamma = \bigcup_{i=1}^{n} \gamma_i$.
We fix a regular neighborhood $N(\gamma_i)$ of $\gamma_i$ in $M$ for every $1 \leqslant i \leqslant n$,
and we denote by $N(\gamma)$ the union of these neighborhoods.
For each $1 \leqslant i \leqslant n$,
we denote by $l_i$ the singular leaf of $\mathcal{F}^{s}$ containing $\gamma_i$
and denote by $s_i$ the slope on $\partial N(\gamma_i)$ represented by
the intersection curves of $l_i$ and $\partial N(\gamma_i)$.
We call $(s_1,\ldots,s_n)$ the \emph{preferred framing} on $\partial N(\gamma)$.
For two slopes $u, v$ on $\partial N(\gamma_i)$ ($i \in \{1, \ldots, n\}$),
let $\Delta(u, v)$ denote the minimal geometric intersection number of $u, v$.

Given any multislope $(r_1,\ldots,r_n)$ on $\partial N(\gamma)$ such that
$r_i \ne s_i$ ($1 \leqslant i \leqslant n$).
We denote by $N$ 
the Dehn filling of $M - Int(N(\gamma))$ along $\partial N(\gamma)$ with the multislope $(r_1,\ldots,r_n)$.
Then $N$ admits a co-orientable taut foliation constructed through the following process:

(1)
We can split open $\mathcal{F}^{s}$ along the union of singular leaves to obtain
an essential lamination $\mathcal{L}^{s}$.
Then every complementary region of $\mathcal{L}^{s}$ in $M$
is a bundle of ideal polygons with an even number of sides over $S^{1}$.
We may assume $\mathcal{L}^{s} \subseteq M - Int(\partial N(\gamma))$,
then we can consider $\mathcal{L}^{s}$ as a lamination of $N$.
The complementary regions of $\mathcal{L}^{s}$ in $N$
are still bundles of ideal polygons with an even number of sides over $S^{1}$.

(2)
We can construct a co-orientable taut foliation of $N$
by filling the complementary regions of $\mathcal{L}^{s}$ in $N$
with monkey saddles (see \hyperref[G2]{[G2]}, or see an explanation in \hyperref[C]{[C, Example 4.22]}).

Considering this taut foliation when $\Delta(r_i,s_i) = 1$,
we have

\begin{prop}\label{Dehn filling}
	Let $(r_1,\ldots,r_n)$ be a multislope on $\partial N(\gamma)$.
	If $\Delta(r_i,s_i) = 1$ for every $1 \leqslant i \leqslant n$,
	then the Dehn filling of $M - Int(N(\gamma))$ along $\partial N(\gamma)$ with the multislope $(r_1,\ldots,r_n)$ 
	admits a co-orientable taut foliation with orderable cataclysm,
	and thus it has left orderable fundamental group.
\end{prop}

\begin{remark}\rm
	Zung (\hyperref[Zu]{[Zu]}) proves that many $3$-manifolds with co-orientable pseudo-Anosov flows have
	left orderable fundamental group.
	Let $\Sigma$ be a closed orientable surface and
	$\phi: \Sigma \to \Sigma$ be a pseudo-Anosov homeomorphism. 
	Assume that 
	the invariant foliations of $\phi$ 
	are co-orientable and that
	$\phi$ preserves their co-orientations.
	Let $N = \Sigma \times S^{1} / \stackrel{\phi}{\sim}$ be 
	the mapping torus of $\phi$,
	and let $X$ be the suspension pseudo-Anosov flow.
	Given a union of closed orbits of $X$ (denoted $\gamma$) and
	a slope on each component of $\gamma$ such that all these slopes have the same sign,
	then the Dehn filling (along $\gamma$) with these slopes has left orderable fundamental group.
	In Proposition \ref{Dehn filling},
	every $r_i$ can be either positive or negative,
	so the resulting manifolds in Proposition \ref{Dehn filling} are not all contained in Zung's result.
\end{remark}

\subsection{Organization}

This paper is organized as follows:

In Section \ref{section 2},
we provide the settings, terminologies,
and backgrounds which are necessary in our discussions.

We prove Theorem \ref{orderable cataclysm} in 
Subsection \ref{subsection 3.1}$\sim$Subsection \ref{subsection 3.2}.
In Subsection \ref{subsection 3.3},
we describe a construction of a left-invariant order of $\pi_1(M)$.
We give a proof for Proposition \ref{Dehn filling} in Subsection \ref{subsection 3.4},

\subsection{Acknowledgements}

The author is indebted to Xingru Zhang for 
many helpful discussions and conversations,
and for his guidance and encouragement.
The author is grateful to Sergio Fenley for the valuable advice and comments.
The author is grateful to Jonathan Zung for reminding him that
Corollary \ref{Anosov} can be known from Thurston's universal circle action,
and the author is grateful to Chi Cheuk Tsang for helpful comments.

\section{Preliminary}\label{section 2}

\subsection{Conventions}

For a set $X$,
we denote by $|X|$ the cardinality of $X$.

For a group $G$ acting on a space $X$,
we will always assume that $G$ acts on $X$ by the left multiplication,
i.e. 
$hg = h \circ g: X \stackrel{g}{\to} X \stackrel{h}{\to} X$ for any
$g,h \in G$.
And we denote by $Stab_G(t)$ the stabilizer of $t$ in $G$,
for every $t \in X$.

\subsection{Taut foliations and group actions on $1$-manifolds}

\begin{nota}\rm
	Suppose that $M$ admits a taut foliation $\mathcal{F}$.
	We will always denote by $\widetilde{\mathcal{F}}$ 
	the pull-back foliation of $\mathcal{F}$ in $\widetilde{M}$,
	and we will always denote by $L(\mathcal{F})$ the leaf space of $\widetilde{\mathcal{F}}$.
	The deck transformations on $\widetilde{M}$ induce an action of $\pi_1(M)$ on $L(\mathcal{F})$,
	called the \emph{$\pi_1$-action} on $L(\mathcal{F})$.
\end{nota}

Note that $L(\mathcal{F})$ is an orientable, connected, simply connected, possibly non-Hausdorff $1$-manifold.
$\widetilde{\mathcal{F}}$ is always co-orientable 
(whether $\mathcal{F}$ is co-orientable or not),
and the orientation on $L(\mathcal{F})$ is induced from
a co-orientation on $\widetilde{\mathcal{F}}$.
If $\mathcal{F}$ is co-orientable,
then the $\pi_1$-action on $L(\mathcal{F})$ is orientation-preserving.
We refer to \hyperref[HR]{[HR]}, \hyperref[Ba2]{[Ba2]} for some references on non-Hausdorff 1-manifolds and
their connections to taut foliations.

In the remainder of this subsection,
we will always assume that $M$ admits a taut foliation $\mathcal{F}$ and
$L = L(\mathcal{F})$.
And we assume further that $L$ is non-Hausdorff.
Now we describe the non-Hausdorff places in $L$
(see Figure \ref{1-manifold} (a) for the local model of such places),
which basically follows from \hyperref[F3]{[F3]} and \hyperref[CD]{[CD, 3.4]}.
We use the term \emph{cataclysm} in \hyperref[CD]{[CD]},
which has the same meaning as the term \emph{branching leaf} in \hyperref[F3]{[F3]}.

\begin{defn}\rm\label{cataclysm}
	Let $\mu$ be a set of points in $L$ with $|\mu| > 1$.
    We call $\mu$ a \emph{cataclysm} of $L$ if:
    
    (1)
    For any $u,v \in \mu$ which are distinct,
    $u, v$ can not be separated by any point in $L -\{u,v\}$.
    
    (2)
    There is no $x \in L - \mu$ such that
    $\mu \cup \{x\}$ satisfies (1).
\end{defn}

Definition \ref{cataclysm} (1) is equivalent to the condition that,
if we choose
an arbitrary neighborhood $U$ of $u$ in $L$ and
an arbitrary neighborhood $V$ of $v$ in $L$,
then
$U \cap V \ne \emptyset$.

\begin{figure}\label{1-manifold}
	\centering
	\subfigure[]{
		\includegraphics[width=0.3\textwidth]{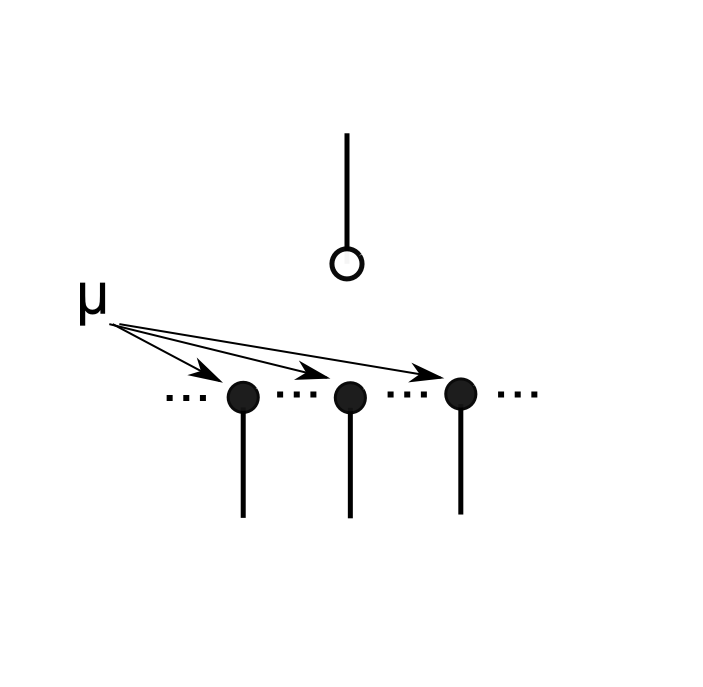}}
	\subfigure[]{
		\includegraphics[width=0.3\textwidth]{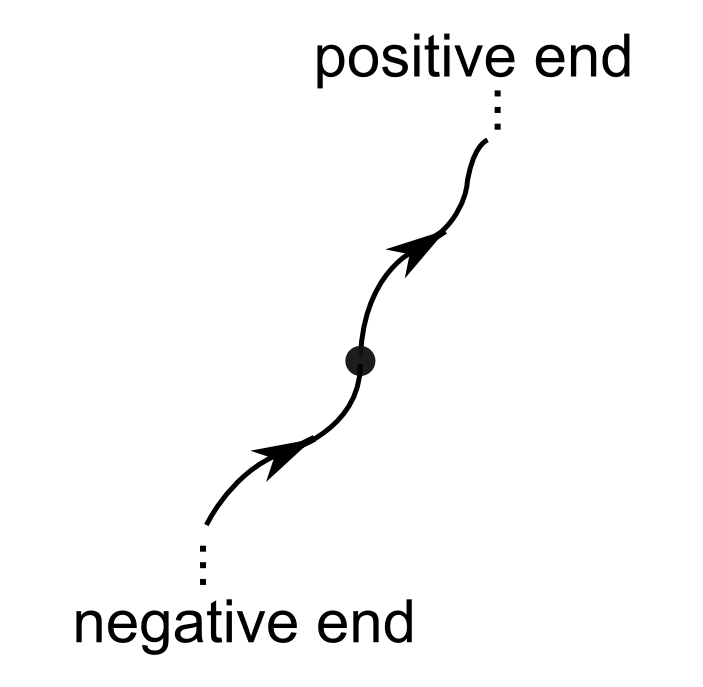}}
	\caption{Suppose that $L = L(\mathcal{F})$ for some taut foliation $\mathcal{F}$ and
	$L$ is non-Hausdorff.
	Picture (a) is the local model of a cataclysm $\mu$ of $L$.
	Fix an orientation on $L$,
	picture (b) describes a positive end and a negative end of $L$:
	the positive end can be represented by a positively oriented ray of $L$,
	and
	the negative end can be represented by a negatively oriented ray of $L$.}
\end{figure}

Figure \ref{1-manifold} (a) is an example of a cataclysm.
Note that for any cataclysm $\mu$ of $L$,
there is a sequence of points $\{t_i\}_{i \in \mathbb{N}}$ in $L$ which
converges to all points in $\mu$ simultaneously,
i.e.
$\mu = \overline{\{t_i \mid i \in \mathbb{N}\}} - \{t_i \mid i \in \mathbb{N}\}$.
See Definition \ref{upward cataclysm} and Remark \ref{remark on upward cataclysm} in
Subsection \ref{subsection 3.1} for more information about cataclysms of $L$.

\begin{defn}\rm\label{cataclysm for foliations}
A \emph{cataclysm} of $\widetilde{\mathcal{F}}$ is
a union of leaves of $\widetilde{\mathcal{F}}$ whose corresponding points in $L$
form a cataclysm in $L$.
\end{defn}

Henceforth,
we will not distinguish between leaves of $\widetilde{\mathcal{F}}$ and
points in $L$.
Any cataclysm of $\widetilde{\mathcal{F}}$ can also be considered as 
a cataclysm of $L$.

Note that any self homeomorphism of $L$ sends a cataclysm to a cataclysm.
	
\begin{defn}\rm\label{orderable cataclysm definition}
	Let $\{g: L \to L \mid g \in G\}$ be an action of a group $G$ on $L$ via homeomorphisms
	(which may not be orientation-preserving).
	We say that $L$ has \emph{orderable cataclysm} 
	(with respect to the action $\{g: L \to L \mid g \in G\}$) if:
	
	$\bullet$
	For any cataclysm $\mu$ of $L$,
	there is a linear order $\stackrel{_\mu}{<}$ on $\mu$ which 
	is preserved by $Stab_G(\mu)$.
\end{defn}

\begin{defn}\rm\label{orderable cataclysm for foliations}
	We say that $\mathcal{F}$ has \emph{orderable cataclysm} if
	$L$ has orderable cataclysm with respect to the $\pi_1$-action on it.
\end{defn}

We may and will always have the following assumption:

\begin{assume}\rm\label{has orderable cataclysm}
	Suppose that $L$ has orderable cataclysm with respect to 
	an action $\{g: L \to L \mid g \in G\}$ of a group $G$ on $L$.
	We fix a linear order $\stackrel{_\mu}{<}$ on every cataclysm 
	$\mu$ of $L$ such that:
	for every $g \in G$ and arbitrary distinct points $p,q$ in $\mu$,
	$g(p) \stackrel{_{g(\mu)}}{<} g(q)$ if and only if $p \stackrel{_\mu}{<} q$.
\end{assume}

Note that not all taut foliations have orderable cataclysm
(see \hyperref[CD]{[CD, Example 3.7]} for an example).
However,
Fenley proves the following theorem in \hyperref[F3]{[F3]}
(see \hyperref[CD]{[CD, Example 3.6]} for an explanation using the terminologies that we adopt):

\begin{thm}[Fenley]\label{pseudo-anosov}
	The stable and unstable foliations of all Anosov flows have
	orderable cataclysm.
\end{thm}

\subsection{Ends of $1$-manifolds}

Throughout this subsection,
we assume that $M$ admits a taut foliation $\mathcal{F}$ and
$L = L(\mathcal{F})$.

By a \emph{ray} of $L$ we mean an embedding $r: [0,+\infty) \to L$
such that 
there is no embedding $f: [0,+\infty) \to L$ with $f(0) = r(0)$ and 
$r([0,+\infty)) \subsetneqq f([0,+\infty))$.
Let $\mathcal{E} = \{\text{rays of } L\}$.
Let $\sim$ be the equivalence relation on $\mathcal{E}$ such that:
for arbitrary two rays $r_1,r_2: [0,+\infty) \to L$,
$r_1,r_2$ are \emph{equivalent} if there are 
$t_1,t_2 \in [0,+\infty)$ with $r_1([t_1,+\infty)) = r_2([t_2,+\infty))$.
Let
$$End(L) = \mathcal{E} / \sim,$$
and we call each element of $End(L)$ an \emph{end} of $L$.
For an end $t$ of $L$ and a ray $r: [0,+\infty) \to L$ which
represents $t$,
we may assume that $t$ is identified with $r(+\infty)$.

\begin{defn}\rm
	Fix an orientation on $L$.
	For every $t \in End(L)$,
	we call $t$ a \emph{positive end} (resp. \emph{negative end}) of $L$ if
	there is a ray $r: [0,+\infty) \to L$ such that
	$r(+\infty) = t$ and the increasing orientation on $r([0,+\infty))$ 
	is consistent with (resp. opposite to) the orientation on $L$,
	see Figure \ref{1-manifold} (b).
	And we denote by 
	$$End_+(L) = \{\text{positive ends of } L\},$$
	$$End_-(L) = \{\text{negative ends of } L\}.$$
\end{defn} 

Because $L$ is a $1$-manifold,
the increasing orientation on an embedded ray in $L$ is necessarily either consistent with or opposite to
the orientation on $L$.
Thus any end of $L$ is either positive or negative.

\begin{remark}\rm
	Let $\{g: L \to L \mid g \in G\}$ be an action of a group $G$ on $L$ via homeomorphisms.
	Then there is an induced action of $G$ on $End(L)$.
	Moreover,
	if the action of $G$ on $L$ is orientation-preserving,
	then every $g \in G$ sends
	positive ends to positive ends and
	sends negative ends to negative ends,
	and therefore $G$ restricts to an action on $End_+(L)$
	(and also restricts to an action on $End_-(L)$).
\end{remark}

The following theorem is implicitly contained in the proof of \hyperref[CD]{[CD, Theorem 3.8]}:

\begin{thm}[Calegari-Dunfield]\label{C-D}
	Let $\{g: L \to L \mid g \in G\}$ be an action of a group $G$ on $L$ via homeomorphisms.
	Suppose that 
	$L$ has orderable cataclysm with respect to the action of $G$.
	Then there is a circular order on $End(L)$ which is preserved by the action of $G$.
\end{thm}

It follows that

\begin{cor}\label{corollary C-D}
	Assume that the conditions of Theorem \ref{C-D} hold.
	Then for arbitrary $x \in End(L)$,
	there is a linear order on $End(L) - \{x\}$ which
	is preserved by $Stab_G(x)$.
\end{cor}

\begin{remark}\rm
	There is a slight difference between \hyperref[CD]{[CD, Theorem 3.8]} and Theorem \ref{C-D}.
	In \hyperref[CD]{[CD, Theorem 3.8]},
	$G$ is the fundamental group of an atoroidal $3$-manifold,
	and $L$ is the order tree associated to $\widetilde{\Lambda}$ for
	some very full genuine lamination $\Lambda$ with orderable cataclysm,
	where $\widetilde{\Lambda}$ denotes the pull-back lamination of $\Lambda$ in the universal cover.
	Calegari and Dunfield prove that
	$G$ acts on the set of ends of $L$ effectively,
	and this action preserves a circular order on the ends of $L$.
	The proof of \hyperref[CD]{[CD, Theorem 3.8]} applies to 
	the setting of Theorem \ref{C-D} since 
	we do not require that $G$ acts on $End(L)$ effectively.
\end{remark}

\subsection{The three types of taut foliations}

Every taut foliation has one of the following three types
(see \hyperref[C]{[C, Definition 4.41]} for example):

\begin{defn}\rm\label{three types}
	For any taut foliation $\mathcal{F}$ of $M$,
	$\mathcal{F}$ has one of the following three types:
	
	$\bullet$
	$\mathcal{F}$ is \emph{$\mathbb{R}$-covered} if
	$L(\mathcal{F})$ is homeomorphic to $\mathbb{R}$.
	
	$\bullet$
	$\mathcal{F}$ has \emph{one-sided branching} if,
	with respect to an orientation on $L(\mathcal{F})$,
	$L(\mathcal{F})$ either has exactly one positive end and more than one negative ends,
	or has exactly one negative end and more than one positive end.
	
	$\bullet$
	$\mathcal{F}$ has \emph{two-sided branching} if,
	with respect to an orientation on $L(\mathcal{F})$,
	$L(\mathcal{F})$ has more than one positive end and more than one negative end.
\end{defn}

We note that $L(\mathcal{F})$ has more than one positive end (resp. negative end) implies that
$L(\mathcal{F})$ has infinitely many positive ends (resp. negative ends).

\begin{lm}
	Let $\mathcal{F}$ be a taut foliation of $M$,
	and we fix an orientation on $L(\mathcal{F})$.
	
	(a)
	If $L(\mathcal{F})$ has more than one positive end,
	then $L(\mathcal{F})$ has infinitely many positive ends.
	
	(b)
	If $L(\mathcal{F})$ has more than one negative end,
	then $L(\mathcal{F})$ has infinitely many negative ends.
\end{lm}

\begin{proof}
	To be convenient,
	we only prove (a).
	Let $G = \pi_1(M)$.
	If $\mathcal{F}$ is not co-orientable,
	then $M$ has a double cover $M^{'}$ such that
	$\mathcal{F}$ pulls-back to a co-orientable taut foliation $\mathcal{F}^{'}$ in $M^{'}$,
	and $L(\mathcal{F}) = L(\mathcal{F}^{'})$ since
	$\widetilde{M}$ is also the universal cover of $M^{'}$ and
	$\widetilde{\mathcal{F}}$ is still the pull-back foliation of $\mathcal{F}^{'}$ in $\widetilde{M}$.
	Thus,
	we may assume that
	$\mathcal{F}$ is co-orientable.
	Then $G$ acts on $L(\mathcal{F})$ via orientation-preserving homeomorphisms.
	
	For each $t \in L(\mathcal{F})$,
	we define $$e(t) = \{a \in End_+(L(\mathcal{F})) \mid \text{there is a ray from } t \text{ to } a\}.$$
	We first claim that there exists $s \in L(\mathcal{F})$ with
	$|e(s)| > 1$.
	
	Suppose otherwise that $|e(t)| = 1$ for all $t \in L(\mathcal{F})$.
	For any $p, q \in L(\mathcal{F})$ such that
	there is a positively oriented path in $L(\mathcal{F})$ from $p$ to $q$,
	we have $e(q) \subseteq e(p)$ and thus $e(q) = e(p)$.
	Let $e_1, e_2$ be distinct positive ends of $L(\mathcal{F})$,
	and let $x, y \in L(\mathcal{F})$ with $e(x) = \{e_1\}, e(y) = \{e_2\}$.
	Note that there are $x_0, x_1, \ldots, x_n \in L(\mathcal{F})$ such that
	$x_0 = x$, $x_n = y$,
	and there is an embedded path in $L(\mathcal{F})$ between $x_i, x_{i+1}$ for every $0 \leqslant i \leqslant n-1$
	(see, for example, \hyperref[GO]{[GO, Construction of tree, method II, Definition 6.9 (3)]}).
	Then $\{e_1\} = e(x_0) = e(x_1) = \ldots = e(x_n) = \{e_2\}$.
	This is a contradiction.
	
	Thus,
	there exists $s \in L(\mathcal{F})$ with $|e(s)| > 1$.
	Let $a, b$ be distinct elements of $e(s)$.
	Let $r_a, r_b$ denote the two rays from $s$ that represents $a, b$ respectively,
	and we choose $u \in r_b - r_a$.
	Then $b \in e(u)$, but $a \notin e(u)$.
	Since $\mathcal{F}$ is a taut foliation,
	there is a simple closed curve in $M$ transverse to $\mathcal{F}$ that intersects all leaves of $\mathcal{F}$.
	This condition implies that
	there is a positively oriented path $\gamma$ in $L(\mathcal{F})$ from $u$ to 
	certain point in the orbit of $u$ (under the $\pi_1$-action) such that its interior contains 
	certain point $s^{'}$ in the orbit of $s$ (under the $\pi_1$-action).
	Let $g \in G$ with $s^{'} = g(s)$.
	Note that $g: L(\mathcal{F}) \to L(\mathcal{F})$ is an orientation-preserving homeomorphism.
	Let $n \in \mathbb{N}$.
	Since $a \in e(s)$ and $g^{n}(a)$ is still a positive end of $L(\mathcal{F})$,
	$g^{n}(a) \in e(g^{n}(s))$.
	Because there is a positively oriented path from $s$ to $u$ and
	a positively oriented path from $u$ to $s^{'} = g(s)$,
	there is a positively oriented path from $g^{i}(s)$ to $g^{i}(u)$ and
	from $g^{i}(u)$ to $g^{i+1}(s)$ for all $i \in \mathbb{Z}$.
	So there is a positively oriented path from $u$ to $g^{n}(s)$ which
	passes through $u, g(s), g(u), \ldots, g^{n-1}(u), g^{n}(s)$.
	It can be deduced from $a \notin e(u)$ that $a \notin e(g^{n}(s))$.
	Combined with $g^{n}(a) \in e(g^{n}(s))$,
	we have $g^{n}(a) \ne a$.
	Hence the positive ends $a, g(a), g^{2}(a), \ldots$ of $L(\mathcal{F})$ are distinct from each other.
\end{proof}

\section{The proof of the main theorems}\label{section 3}

In this section,
we prove Theorem \ref{orderable cataclysm} in Subsection \ref{subsection 3.1}$\sim$\ref{subsection 3.2},
and we describe the construction of a left-invariant order in Subsection \ref{subsection 3.3}.
We prove Proposition \ref{Dehn filling} in Subsection \ref{subsection 3.4}.

Throughout Subsection \ref{subsection 3.1}$\sim$\ref{subsection 3.3},
we have the following assumption:

\begin{assume}\rm\label{F}
$M$ admits 
a co-oriented taut foliation $\mathcal{F}$ which has orderable cataclysm.
\end{assume}

Let $G = \pi_1(M)$,
$L = L(\mathcal{F})$,
and we denote by 
$\{g: L \to L \mid g \in G\}$ the $\pi_1$-action on $L$.
Note that $L$ has orderable cataclysm with respect to $\{g: L \to L \mid g \in G\}$.
As in Assumption \ref{has orderable cataclysm},
every cataclysm $\mu$ of $L$ has a linear order $\stackrel{_\mu}{<}$ which
is preserved by $Stab_G(\mu)$ and is consistent with
the order $\stackrel{_{g(\mu)}}{<}$ for all $g \in G$.

We assume that $L$ has an orientation induced from
the co-orientation on $\mathcal{F}$.
Note that $G$ acts on $L$ via orientation-preserving homeomorphisms.
Moreover,
we assume that $\mathcal{F}$ has two-sided branching,
since Theorem \ref{orderable cataclysm} holds directly when $\mathcal{F}$ is
either $\mathbb{R}$-covered or has one-sided branching
(see \hyperref[Zh]{[Zh]}). 

\subsection{Broken curves in $L$}\label{subsection 3.1}

At first,
we define the sign of every cataclysm in $L$.

\begin{defn}\rm\label{upward cataclysm}
	Let $\mu$ be a cataclysm of $L$ and 
	let $u, v$ be distinct points of $\mu$.
	Notice that $L - \{u\}$ has exactly two components.
	A component of $L - \{u\}$ is called the \emph{upward side}
	(resp. \emph{downward side}) of $u$ if
	there is a positively oriented (resp. negatively oriented) embedded path in $L$
	from $u$ to some point in this component.
	We call $\mu$ an \emph{upward cataclysm} (resp. \emph{downward cataclysm}) if
	$v$ is in the upward side (resp. the downward side) of $u$.
	See Figure \ref{upward and downward} (a), (b) for 
	an example of an upward cataclysm of $L$ and
	an example of a downward cataclysm of $L$, respectively.
\end{defn}

\begin{figure}\label{upward and downward}
	\centering
	\subfigure[]{
		\includegraphics[width=0.3\textwidth]{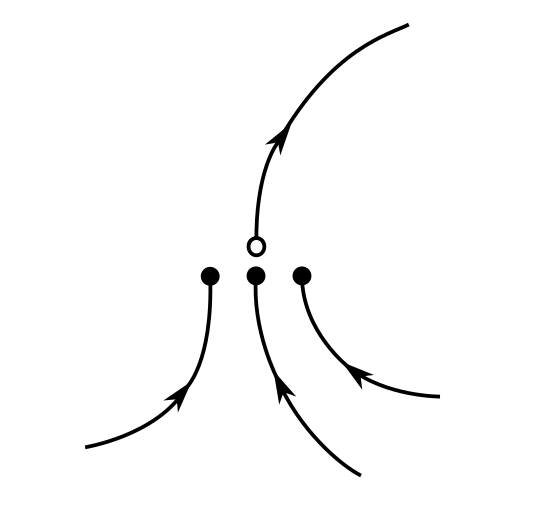}}
	\subfigure[]{
		\includegraphics[width=0.3\textwidth]{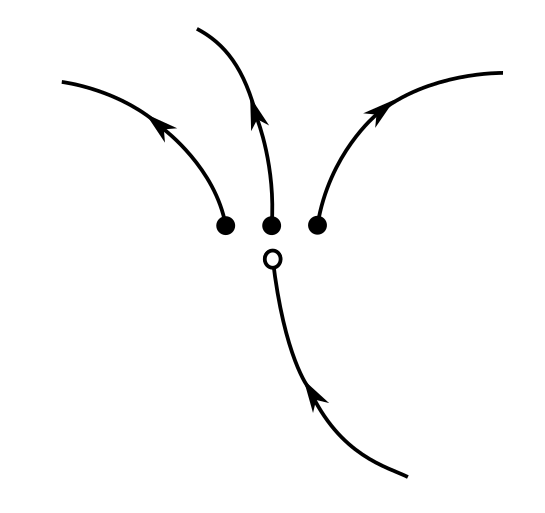}}
	\caption{(a) is the local model of an upward cataclysm in $L$ and 
		(b) is the local model of a downward cataclysm in $L$,
		where $L$ is labeled with the positive orientation.}
\end{figure}

We explain that Definition \ref{upward cataclysm} is independent of the choice of $u$ and $v$.
Because $u$ does not separate arbitrary two distinct points in $\mu - \{u\}$,
all points in $\mu - \{u\}$ are contained in the same side of $L - \{u\}$.
So the choice of $v$ does not affect the sign of $\mu$.
Notice that $u$ is in the upward side of $v$ if and only if
$v$ is in the upward side of $u$.
Thus, 
Definition \ref{upward cataclysm} is also independent of the choice of $u$.

We note that,
if $\mu$ is an upward (resp. downward) cataclysm in $L$,
then there is a sequence of points in $L$ that converges to all elements of $\mu$ simultaneously from
the upward side (resp. downward side) of all of them.

Every point of $L$ is contained in 
no more than one upward cataclysm and no more than one downward cataclysm,
but it may be contained in an upward cataclysm and a downward cataclysm simultaneously.
We explain this as follows.

\begin{remark}\rm\label{remark on upward cataclysm}
	Let $u \in L$ which is contained in some cataclysm.
	Let $a,b \in L$ such that 
	$a, u$ are both contained in a cataclysm and
	$b, u$ are also both contained in a cataclysm.
	Then every point in $L - \{a,b,u\}$ does not separate both $\{a,u\}$ and $\{b,u\}$,
	and thus does not separate $\{a,b\}$.
	It follows that
	$a,b$ are contained in the same cataclysm if and only if
	$u$ does not separate $\{a,b\}$,
	and therefore if and only if $a,b$ are contained in 
	the same side of $u$.
	Hence $u$ is contained in no more than one upward cataclysm and
	no more than one downward cataclysm.
	However,
	we can not exclude the case where
	$u$ is contained in an upward cataclysm and a downward cataclysm simultaneously.
\end{remark}

Let $u,v$ be distinct points in $L$.
Now we describe a uniquely defined ``broken path'' $\alpha(u,v)$ from $u$ to $v$,
which consists of all points in
$$\{t \in L - \{u,v\} \mid t \text{ separates } u \text{ and } v\} \cup \{u,v\}.$$
Our description of $\alpha(u,v)$ basically follows from \hyperref[CD]{[CD, 6.20]}
(compare with \hyperref[RS]{[RS, Definition 3.5, Theorem 3.6]}).

Fix $u,v$ given above.
There is a unique sequence of points $t_1,\ldots,t_n \in L$ such that
the following conditions hold (see Figure \ref{broken path} (a), (b)):

\begin{figure}\label{broken path}
	\centering
	\subfigure[]{
		\includegraphics[width=0.7\textwidth]{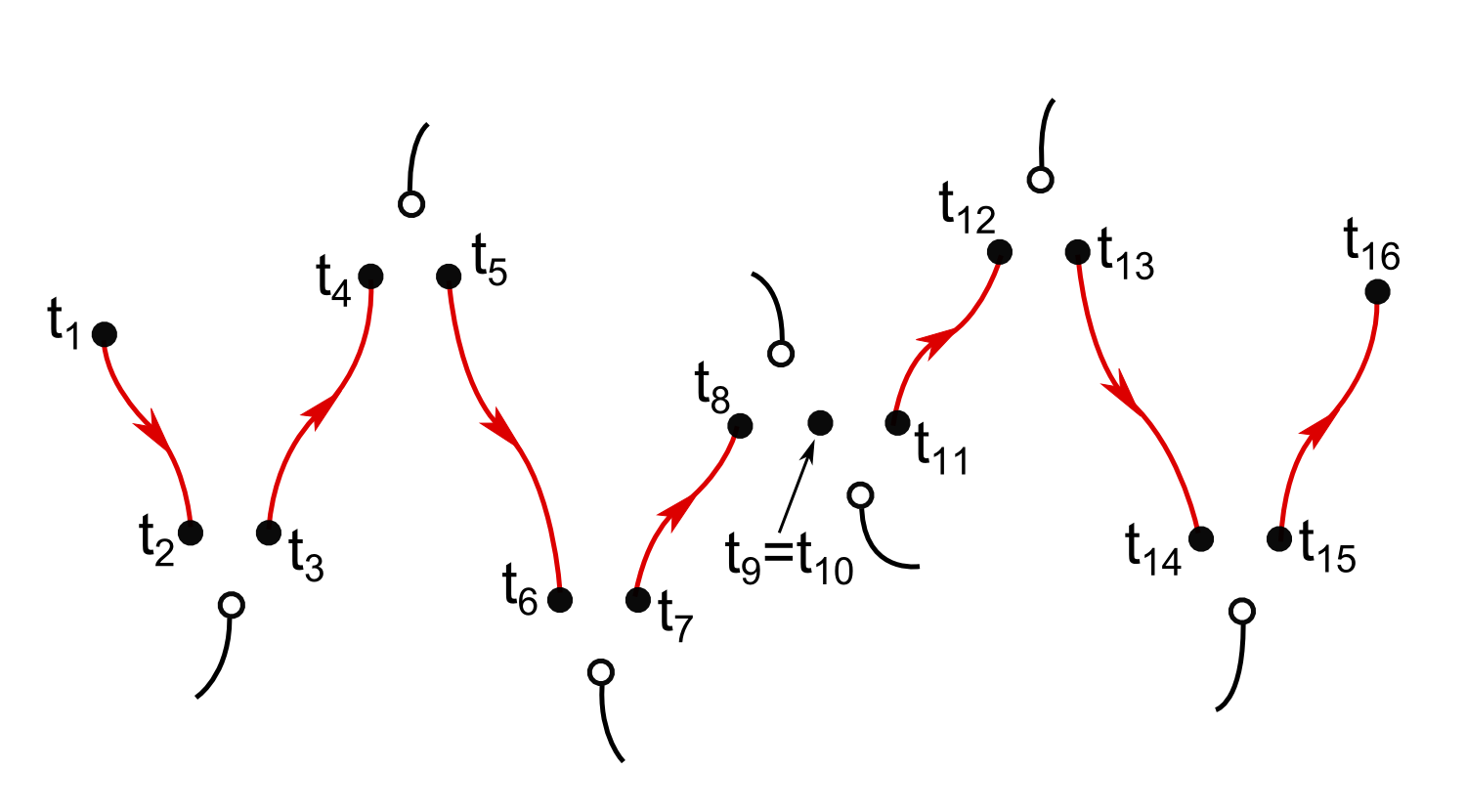}}
	\subfigure[]{
		\includegraphics[width=0.7\textwidth]{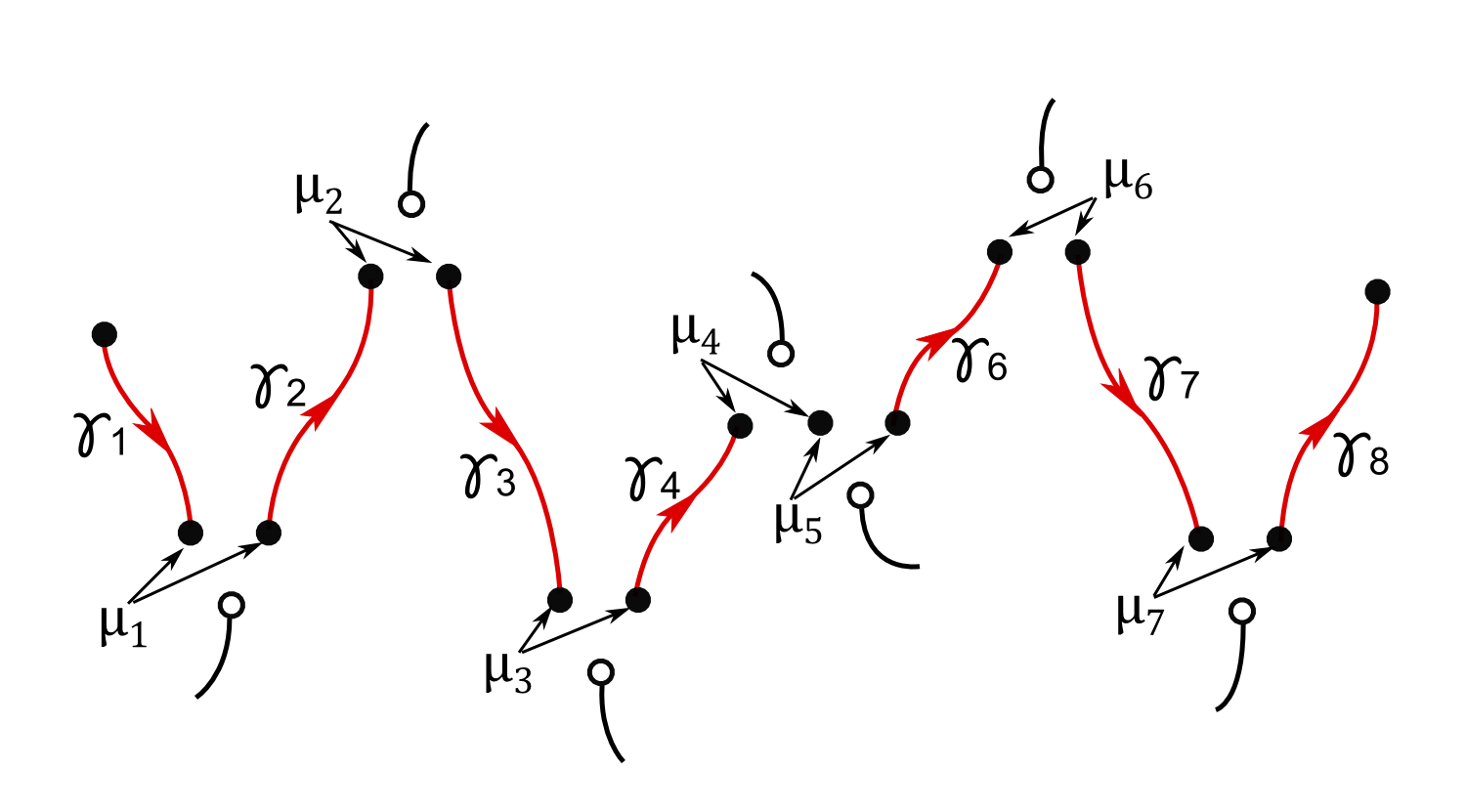}}
	\caption{We give an example of $\alpha(u,v)$ 
		(where $u = t_1$ and $v = t_{16}$ in the picture).
	The sequence of points $\{t_i\}_{1 \leqslant i \leqslant 16}$ is shown in picture (a),
	where $t_9 = t_{10}$.
	The paths $\{\gamma_k\}_{1 \leqslant k \leqslant 8}$ and 
	cataclysms $\{\mu_k\}_{1 \leqslant k \leqslant 7}$ are shown in picture (b),
	where $\gamma_5$ is a trivial path.
	In the sequence $\gamma_1,\mu_1,\gamma_2,\ldots,\mu_7,\gamma_8$,
	all $\gamma_j, \mu_k$ have the relations as given in Condition \ref{condition 4}.}
\end{figure}

\begin{condition}\rm\label{condition 1}
	$t_1 = u$,
	$t_n = v$,
	and $t_{n-1} \ne v$.
\end{condition}

\begin{condition}\rm\label{condition 2}
For every pair $(t_{2k-1}, t_{2k})$,
there is an embedded path $\gamma_k$ in $L$ from $t_{2k-1}$ to $t_{2k}$
(we allow that
$t_{2k-1} = t_{2k}$ and $\gamma_k$ is a trivial path).
\end{condition}

\begin{condition}\rm\label{condition 3}
For every pair $(t_{2k}, t_{2k+1})$,
$t_{2k}, t_{2k+1}$ are distinct points contained in the same cataclysm.
We denote this cataclysm by $\mu_k$.
\end{condition}

\begin{condition}\rm\label{condition 4}
For every $k$,
if $\mu_k$ is an upward cataclysm,
then $\mu_{k+1}$ is a downward cataclysm,
and vice versa.
Moreover,
if $\mu_k$ is an upward cataclysm,
then $\gamma_k$ is positively oriented in $L$ and
$\gamma_{k+1}$ is negatively oriented in $L$.
If $\mu_k$ is a downward cataclysm,
then $\gamma_k$ is negatively oriented in $L$ and
$\gamma_{k+1}$ is positively oriented in $L$.
Here,
a trivial path is both positively oriented
and negatively oriented in $L$.
\end{condition}

Let $\alpha(u,v)$ be the broken path that
starts at $u = t_1$ and goes along $\gamma_1$ to $t_2$,
jumps to $t_3$ at $\mu_1$ and goes along $\gamma_2$ to $t_4$,
jumps to $t_5$ at $\mu_2$ and goes along $\gamma_3$ to $t_6$,
and so on inductively,
until it ends at $t_n = v$.
See Figure \ref{broken path} for an example of $\alpha(u,v)$.

Note that $\alpha(u,v)$ is contained in all paths from $u$ to $v$,
and it passes through all points in $L - \{u,v\}$ that separate $u$ and $v$.

\begin{defn}\rm\label{cusped path}
For $u,v,\{t_i\}_{1 \leqslant i \leqslant n},\{\mu_k\}_{1 \leqslant k \leqslant [\frac{n-1}{2}]}$ given above,
every pair $(t_{2k}, t_{2k+1})$ is called a \emph{cusp} of $\alpha(u,v)$.
Recall that $\mu_k$ is the cataclysm that contains both of $t_{2k}, t_{2k+1}$,
and recall from Assumptions \ref{has orderable cataclysm} and \ref{F} that,
$\mathcal{F}$ has orderable cataclysm,
and there is a linear order ``$\stackrel{_{\mu_k}}{<}$'' associated to $\mu_k$.
We call $(t_{2k}, t_{2k+1})$ a \emph{positive cusp} if $t_{2k} \stackrel{_{\mu_k}}{<} t_{2k+1}$,
and call $(t_{2k}, t_{2k+1})$ a \emph{negative cusp} if $t_{2k} \stackrel{_{\mu_k}}{>} t_{2k+1}$.
\end{defn}

\begin{convention}\rm\label{trivial path orientation}
	Recall that for every trivial path $\gamma_k$,
	it is both positively oriented and negatively oriented in $L$.
	We fix a unique orientation on $\gamma_k$ which satisfies Condition \ref{condition 4}:
	
	$\bullet$
	We assume that
	$\gamma_k$ is positively oriented and not negatively oriented in $L$ if 
	$\mu_{k-1}$ is a downward cataclysm or $\mu_k$ is an upward cataclysm.
	
	$\bullet$
	We assume that
	$\gamma_k$ is negatively oriented and not positively oriented in $L$ if 
	$\mu_{k-1}$ is an upward cataclysm or $\mu_k$ is a downward cataclysm.
\end{convention}

Henceforth,
every trivial path $\gamma_k$ has a unique orientation.
We explain the meaning of the orientation on $\gamma_k$ as follows.
Let $t$ denote $t_{2k-1} = t_{2k}$, which is the single point in $\gamma_k$.
Assume without loss of generality that 
$\mu_{k-1}$ is a downward cataclysm and $\mu_k$ is an upward cataclysm.
Then $\mu_{k-1} - \{t\}$ is contained in the downward side of $t$, 
and $\mu_k - \{t\}$ is contained in the upward side of $t$.
Thus,
$\gamma_k$ is a positively oriented path in the sense that
it goes from $\mu_{k-1}$ to $\mu_k$,
or goes from the downward side of $t$ to the upward side of $t$.

It follows from Condition \ref{condition 4} that

\begin{fact}\rm
The orientations on $\{\gamma_k\}_{1 \leqslant k \leqslant [\frac{n}{2}]}$ change at every cusp.
\end{fact}

\begin{figure}\label{cusped}
	\centering
	\subfigure[]{
		\includegraphics[width=0.4\textwidth]{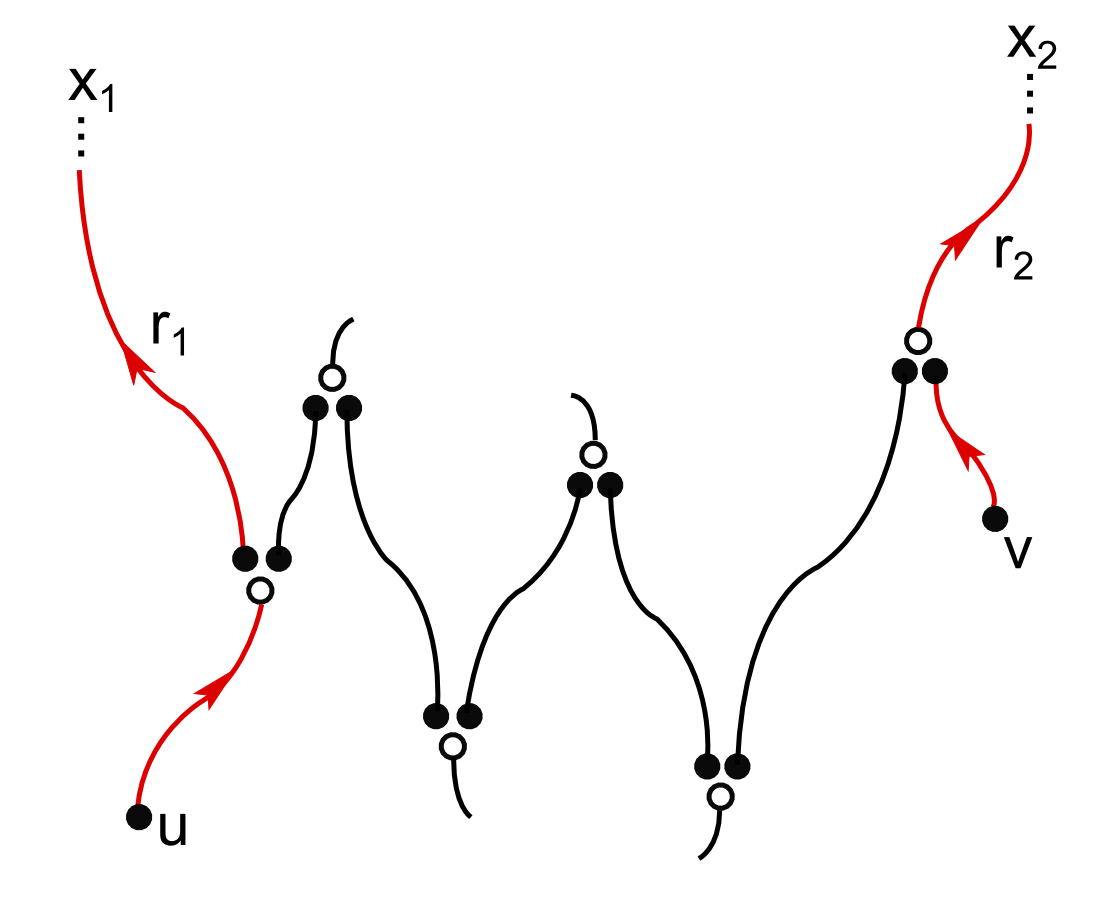}}
	\subfigure[]{
		\includegraphics[width=0.4\textwidth]{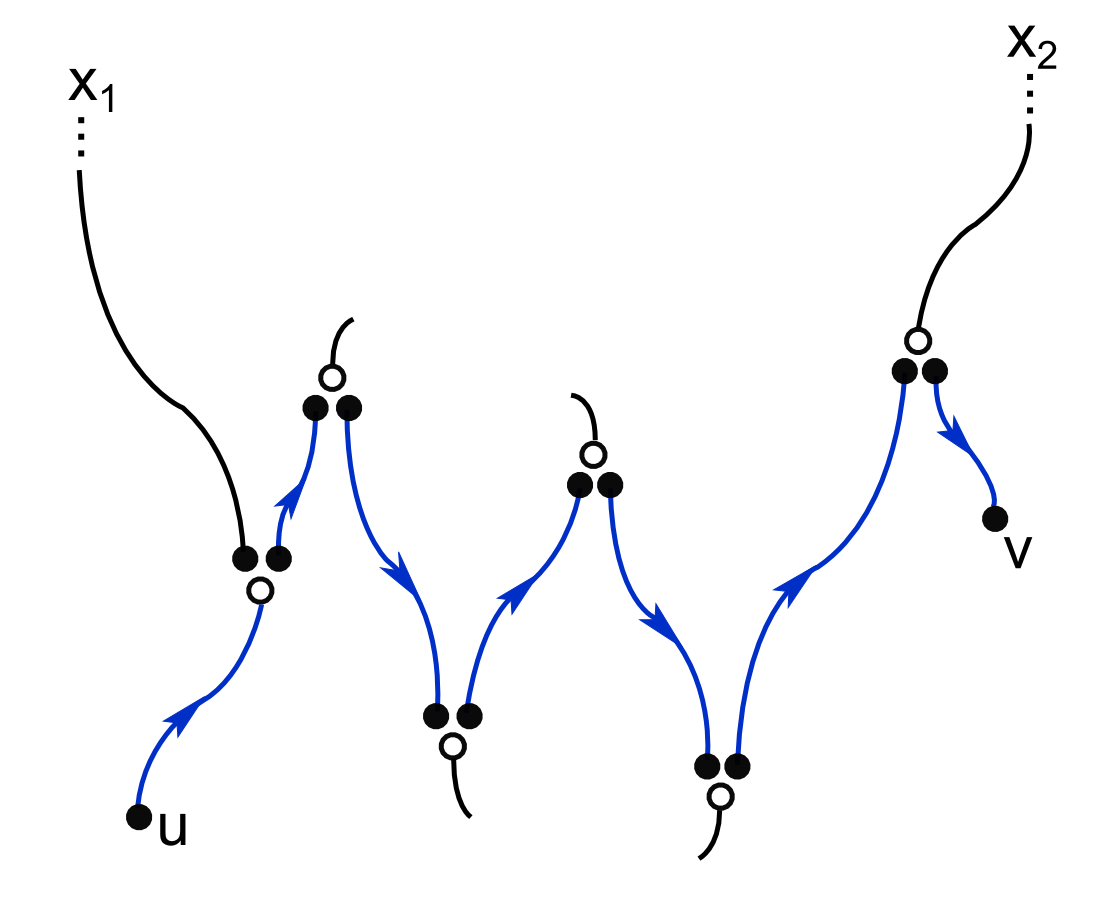}}
	\subfigure[]{
		\includegraphics[width=0.4\textwidth]{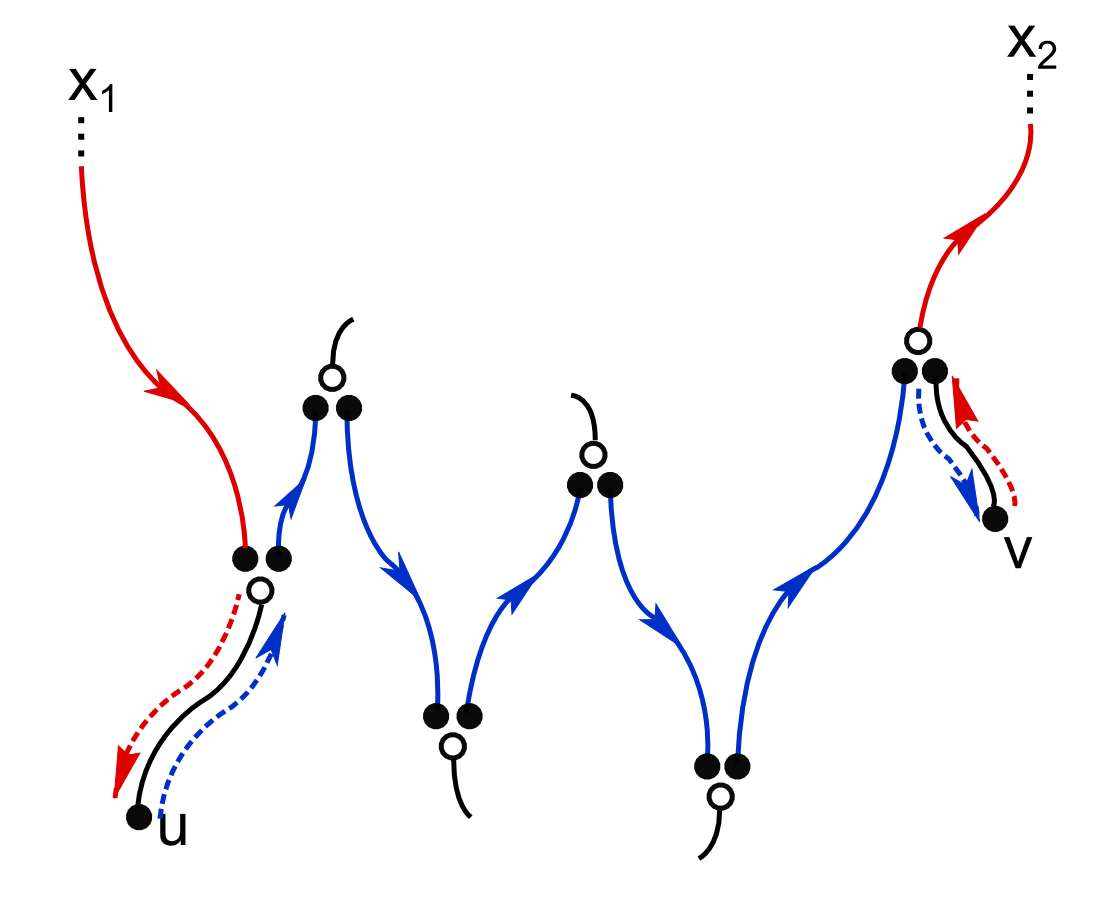}}
	\subfigure[]{
		\includegraphics[width=0.4\textwidth]{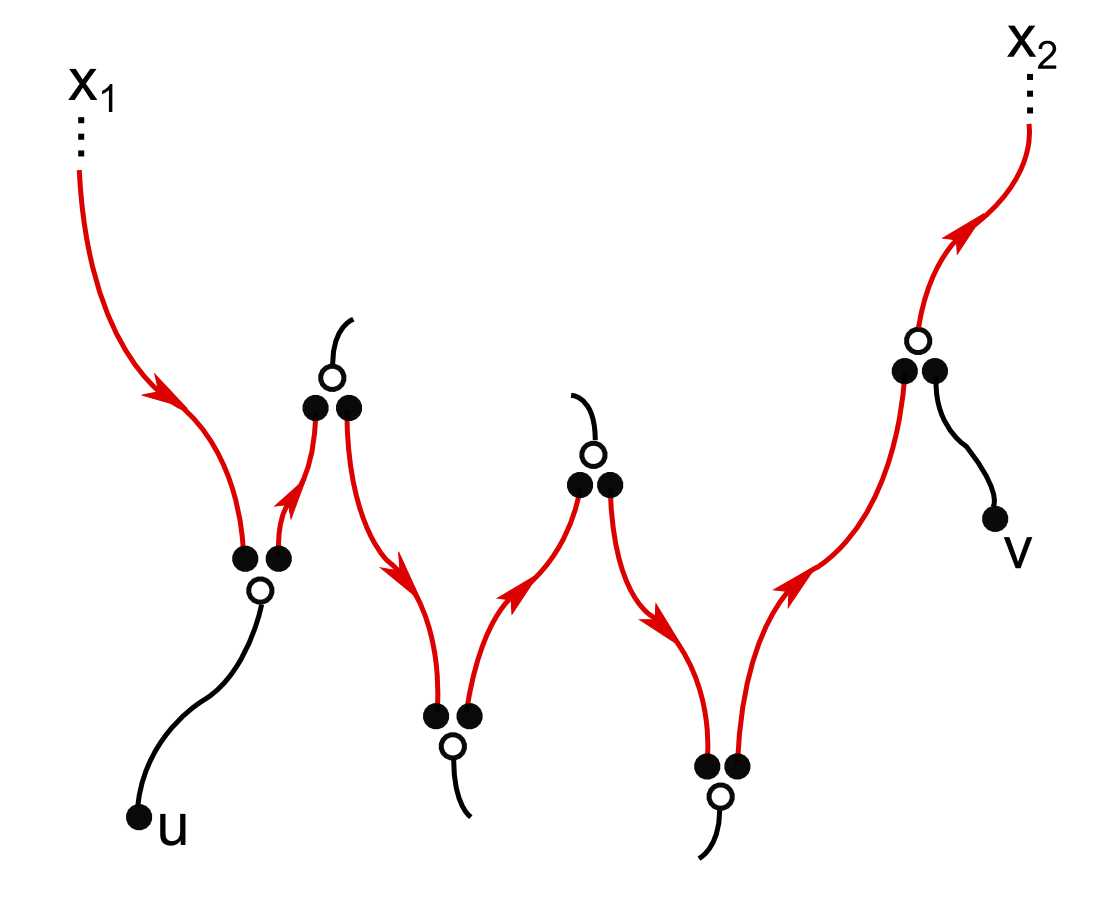}}
	\caption{We give an example of $\alpha(x_1,x_2)$ for two distinct ends $x_1, x_2$ of $L$.
		In picture (a),
		there are two rays $r_1,r_2$ which represent $x_1, x_2$ respectively and do not intersect,
		and we let $u = r_1(0)$, $v = r_2(0)$.
		The blue broken path in picture (b) is $\alpha(u,v)$.
		Picture (c) describes the broken curve that goes from $x_1$ to $x_2$ along
		$\overline{r_1}$ (i.e. $r_1$ with inverse direction), $\alpha(u,v)$ and $r_2$,
		which has some self-intersections.
		The red broken curve in picture (d) is $\alpha(x_1,x_2)$,
		obtained from deleting the self-intersections of the previous broken curve.}
\end{figure}

In the following construction,
we define an embedded broken curve between any two distinct ends of $L$:

\begin{construction}\rm\label{cusped line}
	Let $x_1,x_2$ be distinct ends of $L$.
	We choose rays $r_1,r_2: [0,+\infty) \to L$ which represent
	$x_1, x_2$ respectively,
	and we may identify $x_1$, $x_2$ with $r_1(+\infty)$, $r_2(+\infty)$.
	We may assume that $r_1, r_2$ do not intersect.
	Let $u = r_1(0)$, $v = r_2(0)$.
	We choose a broken curve (which may have self-intersections) that
	(1)
	starts at $x_1$ and goes along the inverse direction of $r_1$ to $u$,
	(2)
	goes from $u$ along $\alpha(u,v)$ to $v$,
	(3)
	goes from $v$ to $x_2$ along $r_2$.
	We delete all self-intersections of this broken curve.
	Then we obtain an embedded broken curve from $x_1$ to $x_2$.
	We denote by $\alpha(x_1,x_2)$ this embedded broken curve.
	Since $L$ is simply connected, 
	$\alpha(x_1,x_2)$ is uniquely defined,
	i.e. it is independent of the choices of rays representing $x_1, x_2$.
\end{construction}

See Figure \ref{cusped} (a)$\sim$(d) for an example of the process in Construction \ref{cusped line}.

We note that from Construction \ref{cusped line},
there is a unique sequence of points $t_1,\ldots,t_n$ with $2 \mid n$ such that
$t_1 = x_1$, 
$t_n = x_2$
and they satisfy Condition \ref{condition 2}$\sim$\ref{condition 4} in the previous discussions.
We still denote by $\gamma_k$ 
the embedded path in $L$ from $t_{2k-1}$ to $t_{2k}$ (which may be a trivial path),
and we still denote by $\mu_k$ 
the cataclysm that contains both of $t_{2k}, t_{2k+1}$.
Then $\alpha(x_1,x_2)$ is the broken curve that starts at $x_1$,
goes along $\gamma_1$ to $t_2$,
jumps to $t_3$ at $\mu_1$ and goes along $\gamma_2$ to $t_4$,
and so on inductively,
goes along $\gamma_{\frac{n}{2}}$ and ends at $t_n = x_2$.
We call every $\gamma_k$ a \emph{segment} of $\alpha(x_1,x_2)$.
And we call $\gamma_1$ the \emph{first segment} of $\alpha(x_1,x_2)$,
$\gamma_{\frac{n}{2}}$ the \emph{last segment} of $\alpha(x_1,x_2)$.

Similar to Definition \ref{cusped path},
we call every pair $(t_{2k}, t_{2k+1})$ a \emph{cusp} of $\alpha(x_1,x_2)$,
and call $(t_{2k}, t_{2k+1})$ a \emph{positive cusp} (resp. \emph{negative cusp}) if
$t_{2k} \stackrel{_{\mu_k}}{<} t_{2k+1}$ (resp. $t_{2k} \stackrel{_{\mu_k}}{>} t_{2k+1}$).
And we fix a unique orientation on every $\gamma_k$ which is a trivial path,
in the sense of Convention \ref{trivial path orientation}.
Then the orientations on the segments of $\alpha(x_1,x_2)$ change at every cusp.

\begin{defn}\rm
	For distinct positive ends $x_1, x_2$ of $L$,
	we define
	$$n(x_1,x_2) = |\{\text{positive cusps of } \alpha(x_1,x_2)\}| - 
	|\{\text{negative cusps of } \alpha(x_1,x_2)\}|.$$
\end{defn}

It's clear that

\begin{lm}\label{anti-sym}
	$n(x_1,x_2) = -n(x_2,x_1)$.
\end{lm}

Furthermore,

\begin{lm}\label{nonzero}
	Let $x_1,x_2$ be distinct positive ends of $L$.
	Then $n(x_1,x_2) \ne 0$.
\end{lm}
\begin{proof}
	Let $m = |\{\text{cusps of } \alpha(x_1,x_2)\}|$. 
	Since $x_1,x_2$ are positive ends of $L$,
	the first segment of $\alpha(x_1,x_2)$ is negatively oriented in $L$,
	and the last segment of $\alpha(x_1,x_2)$ is positively oriented in $L$.
	Notice that the orientations on the segments of $\alpha(x_1,x_2)$ change at every cusp.
	So
	$2 \nmid m$.
	
	We have $n(x_1,x_2) = m - 2 \cdot |\{\text{negative cusps of } \alpha(x_1,x_2)\}|$,
	and thus
	$2 \nmid n(x_1,x_2)$.
	So $n(x_1,x_2) \ne 0$.
\end{proof}

\subsection{The linear order on $End_+(L)$}\label{subsection 3.2}

In this subsection,
we aim to define a linear order on $End_+(L)$ preserved by the action of $G$.
Our first step is the following lemma:

\begin{lm}
	Let $x_1,x_2,x_3$ be three distinct positive ends of $L$.
	Then
	$$n(x_1,x_3) \in \{n(x_1,x_2)+n(x_2,x_3)-1,n(x_1,x_2)+n(x_2,x_3)+1\}.$$
\end{lm}
\begin{proof}
    Let $\beta_1 = \alpha(x_1,x_2) \cap \alpha(x_1,x_3)$,
    $\beta_2 = \alpha(x_1,x_2) \cap \alpha(x_2,x_3)$,
    $\beta_3 = \alpha(x_1,x_3) \cap \alpha(x_2,x_3)$.
    See Figure \ref{beta}.
    
    Let $t \in \alpha(x_1,x_2) \cup \alpha(x_1,x_3) \cup \alpha(x_2,x_3)$.
    For each $i \in \{1,2,3\}$,
    we say $x_i$ is contained in a component of $L  - \{t\}$ if
    there is a ray of $L$ representing the end $x_i$ that is contained in this component.
    Note that $t \in \beta_i$ if and only if $x_i$ is contained in one component of $L - \{t\}$ and
    $\{x_1, x_2, x_3\} - \{x_i\}$ is contained in the other component of $L - \{t\}$.
    So $\beta_1, \beta_2, \beta_3$ are disjoint from each other.
    In addition,
    assume further $t \in \alpha(x_i, x_j)$ ($1 \leqslant i < j \leqslant 3$),
    then $x_i, x_j$ are contained in the two components of $L - \{t\}$ respectively.
    We denote by $L_i(t), L_j(t)$ these two components of $L - \{t\}$,
    where $x_i$ is contained in $L_i(t)$ and $x_j$ is contained in $L_j(t)$.
    Let $k = \{1,2,3\} - \{i,j\}$.
    Then $t \in \beta_i$ when $x_k$ is contained in $L_j(t)$, 
    and
    $t \in \beta_j$ when $x_k$ is contained in $L_i(t)$.
    Thus,
    $t$ is contained in either $\beta_i$ or $\beta_j$.
    
    We assume that the orientations on $\beta_1,\beta_3$ are induced from $\alpha(x_1,x_3)$,
    and the orientation on $\beta_2$ is induced from $\alpha(x_1,x_2)$.
    Then
    $$\alpha(x_1,x_3) = \beta_3 \circ \beta_1,$$
    $$\alpha(x_1,x_2) = \beta_2 \circ \beta_1,$$
    $$\alpha(x_2,x_3) = \beta_3 \circ \overline{\beta_2}.$$
    Here,
    $\overline{\beta_2}$ means the path $\beta_2$ with opposite orientation,
    and $\circ$ means the composition of paths.
    By the above equations,
    we can ensure that:
    
\begin{figure}\label{beta}
	\includegraphics[width=0.6\textwidth]{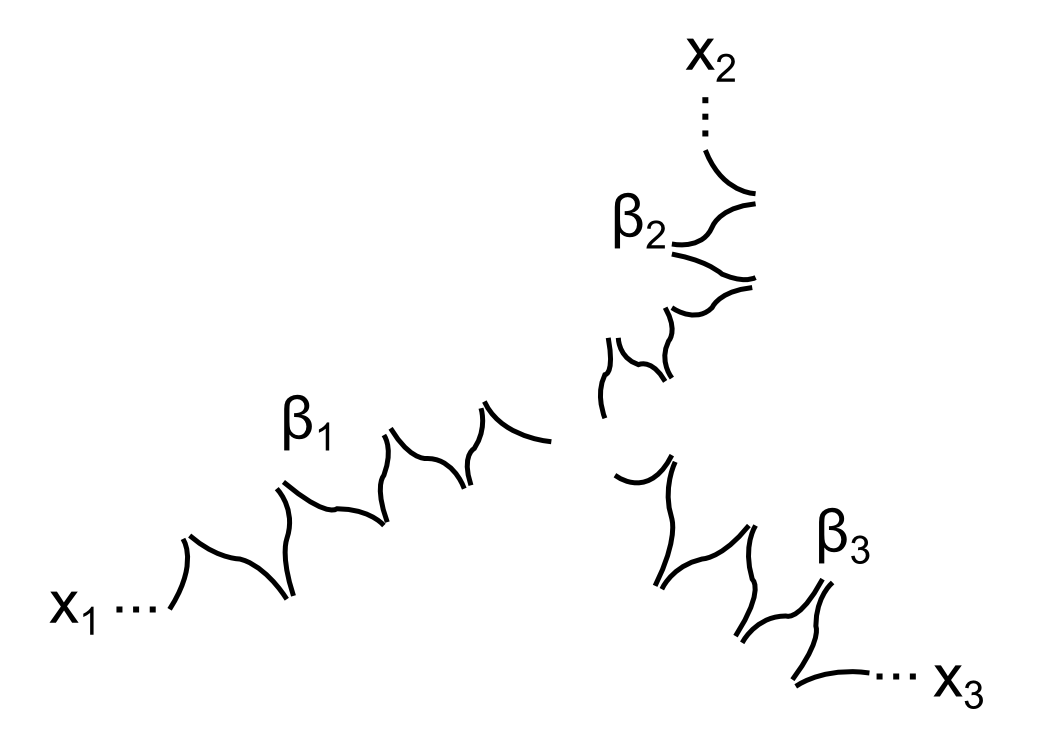}
	\caption{The pictures of $\beta_1$, $\beta_2$, $\beta_3$.
		Here, the ``common intersection place'' of $\beta_1$, $\beta_2$, $\beta_3$ is undetermined
		(it has more than one cases).}
\end{figure}
    
    $\bullet$
    Every cusp contained in $\beta_1$ appears one time in $\alpha(x_1,x_3)$ and 
    $\alpha(x_1,x_2)$,
    and appears zero time in $\alpha(x_2,x_3)$.
    As a result,
    it appears one time in both $\alpha(x_1,x_3)$ and
    $\alpha(x_1,x_2) \cup \alpha(x_2,x_3)$.
    
    $\bullet$
    Every cusp contained in $\beta_2$ appears one time in $\alpha(x_1,x_2)$,
    negative one time in $\alpha(x_2,x_3)$ (i.e. one time with opposite sign),
    and zero time in $\alpha(x_1,x_3)$.
    Therefore,
    every cusp contained in $\beta_2$ contributes zero in both $n(x_1,x_3)$ and
    $n(x_1,x_2) + n(x_2,x_3)$.
    
    $\bullet$
    Every cusp contained in $\beta_3$ appears one time in $\alpha(x_1,x_3)$ and 
    $\alpha(x_2,x_3)$,
    and appears zero time in $\alpha(x_1,x_2)$.
    Hence
    it appears one time in both $\alpha(x_1,x_3)$ and
    $\alpha(x_1,x_2) \cup \alpha(x_2,x_3)$.
    
    By the above discussions,
    all cusps contained in $\beta_1,\beta_2,\beta_3$ contribute equally to 
    both $n(x_1,x_3)$ and $n(x_1,x_2)+n(x_2,x_3)$,
    and thus we can ignore them.
    It only remains to consider the cusps of
    $\alpha(x_1,x_2)$, $\alpha(x_1,x_3)$, $\alpha(x_2,x_3)$ that
    appear in their ``common intersection place''.
    
    For a cusp $\mu$ in $\alpha(x_i, x_j)$ ($1 \leqslant i < j \leqslant 3$),
    we call $\mu$ a ``special cusp'' if 
    $(\mu \nsubseteq \beta_i) \wedge (\mu \nsubseteq \beta_j)$,
    where ``$\wedge$'' denotes the logical conjunction ``and''.
    Note that $\mu$ is also not contained in $\beta_k$ 
    (where $k = \{1,2,3\} - \{i,j\}$), 
    as
    $\alpha(x_i, x_j) = \beta_i \cup \beta_j$ and
    $\beta_k \cap \beta_i = \beta_k \cap \beta_j = \emptyset$.
    In addition,
    a special cusp is contained in exactly one of
    $\alpha(x_1,x_2)$, $\alpha(x_1,x_3)$, $\alpha(x_2,x_3)$:
    if a cusp is contained in two of them,
    for instance $\alpha(x_1,x_2), \alpha(x_1,x_3)$,
    then this cusp must be contained in $\alpha(x_1,x_2) \cap \alpha(x_1,x_3) = \beta_1$ and thus is not a special cusp.
    Also,
    each of $\alpha(x_1,x_2)$, $\alpha(x_1,x_3)$, $\alpha(x_2,x_3)$ contains no more than one special cusp
    (otherwise, 
    assume $\alpha(x_i,x_j)$ contains two special cusps, 
    then at least one of them is contained in either $\beta_i$ or $\beta_j$,
    which is impossible).
    
    There is a point $t \in \alpha(x_1,x_3)$ such that:
    $\beta_1$ is contained in
    the subpath of $\alpha(x_1,x_3)$ from $x_1$ to $t$,
    and
    $\beta_3$ is contained in
    the subpath of $\alpha(x_1,x_3)$ from $t$ to $x_3$.
    Here, 
    $t$ may not be uniquely determined by these conditions
    (if there is a special cusp in $\alpha(x_1,x_3)$,
    then both of the two points of this cusp satisfy these conditions).
    We have the following three cases:
    
    \begin{figure}\label{cusp}
    	\centering
    	\subfigure[]{
    		\includegraphics[width=0.4\textwidth]{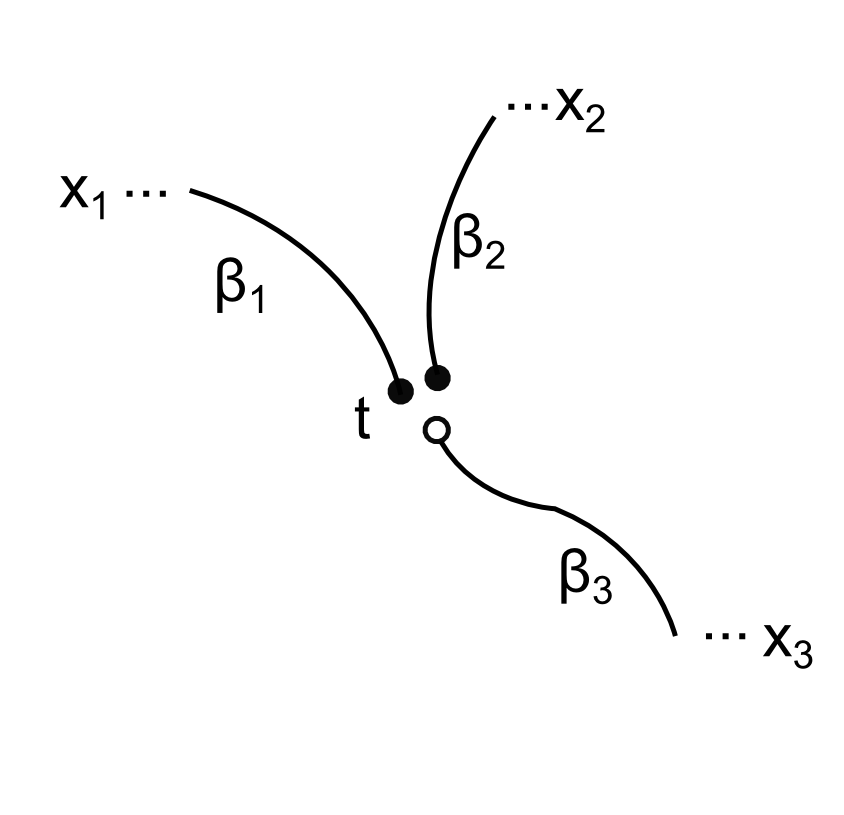}}
    	\subfigure[]{
    		\includegraphics[width=0.4\textwidth]{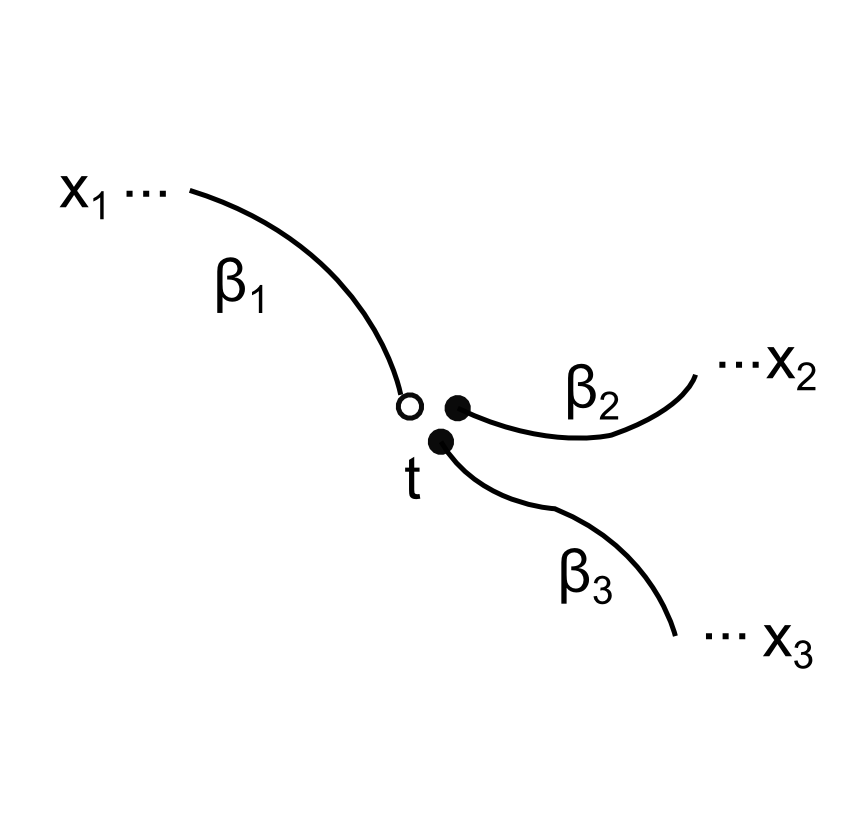}}
    	\subfigure[]{
    		\includegraphics[width=0.4\textwidth]{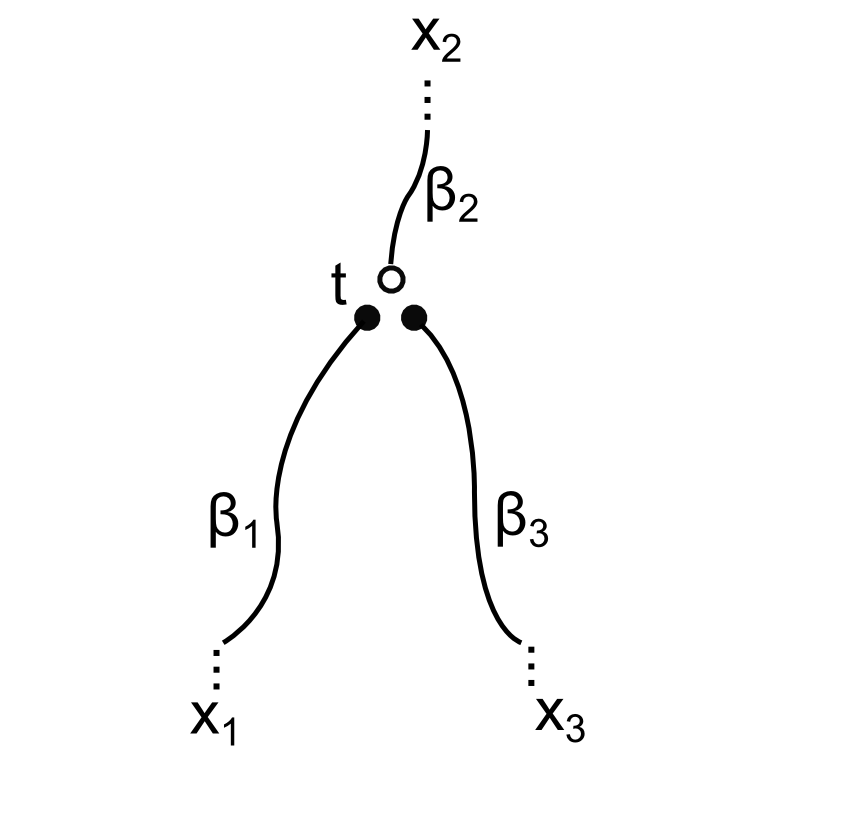}}
    	\subfigure[]{
    		\includegraphics[width=0.4\textwidth]{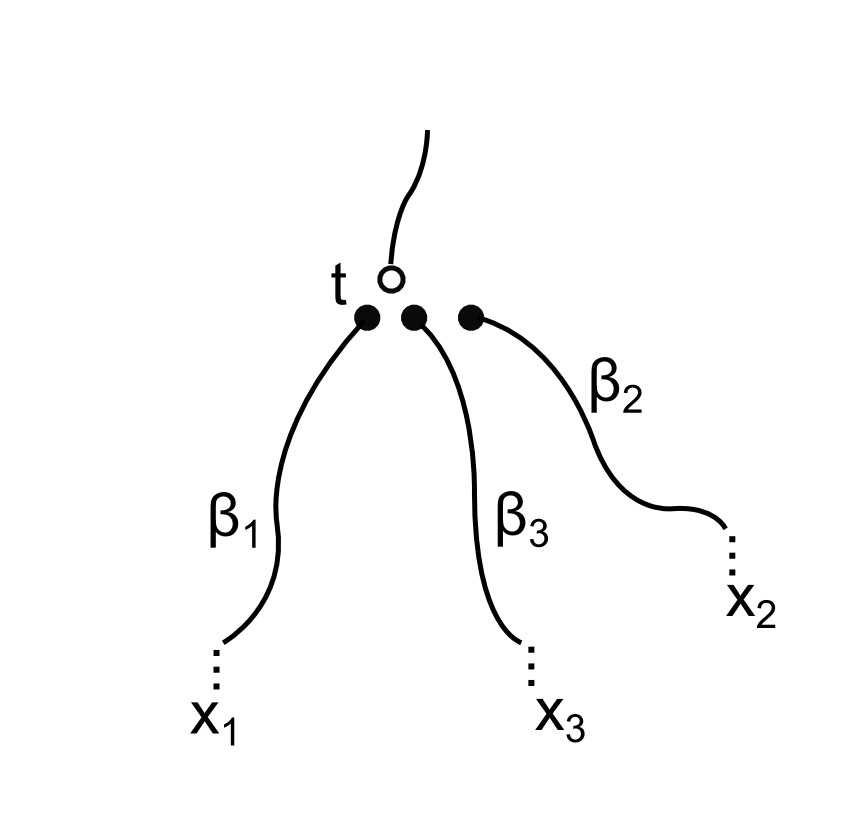}}
    	\caption{}
    \end{figure}
	
	\begin{case}\rm
		Assume that there is no special cusp in $\alpha(x_1,x_3)$,
		which implies that
		every cusp of $\alpha(x_1,x_3)$ is contained in 
		either $\beta_1$ or $\beta_3$.
		Then one of the following two possibilities holds:
		
		$\bullet$
		$t \in \beta_1$.
		Then there is a special cusp of $\alpha(x_1,x_2)$ that contains $t$
		(see Figure \ref{cusp} (a)).
		We denote by $\mu$ the cataclysm that contains this special cusp.
		Then $\alpha(x_2,x_3)$ does not have a cusp at $\mu$.
		
		$\bullet$
		$t \in \beta_3$.
		Then $t$ is contained in a special cusp of $\alpha(x_2,x_3)$ 
		(see Figure \ref{cusp} (b)).
		We denote by $\mu$ the cataclysm that contains this special cusp.
		Then $\alpha(x_1,x_2)$ has no cusp at $\mu$.
		
		Thus,
		there is a cataclysm $\mu$ such that
		(1)
		$t \in \mu$,
		(2)
		exactly one of $\alpha(x_1,x_2), \alpha(x_2,x_3)$ has a cusp at $\mu$,
		(3)
		$\alpha(x_1,x_3)$ has no cusp at $\mu$.
		It follows that
		$|n(x_1,x_3)- (n(x_1,x_2)+n(x_2,x_3))| = 1$.
	\end{case}
	
	\begin{case}\rm
		Assume that there is a special cusp in $\alpha(x_1,x_3)$.
		Then $t$ is contained in this cusp,
		and we may assume without loss of generality $t \in \beta_1$.
		Assume further that
		$t$ is not contained in a special cusp of $\alpha(x_1,x_2)$
		(see Figure \ref{cusp} (c)).
		Then there is a cataclysm $\mu$ such that
		(1)
		$t \in \mu$,
		(2)
		$\alpha(x_1,x_3)$ has a cusp at $\mu$,
		(3)
		both of $\alpha(x_1,x_2), \alpha(x_2,x_3)$ have no cusp at $\mu$.
		Thus
		$|n(x_1,x_3)- (n(x_1,x_2)+n(x_2,x_3))| = 1$.
	\end{case}
	
	\begin{case}\rm
		We assume that there is a special cusp of $\alpha(x_1,x_3)$ 
		(then $t$ is contained in this cusp).
		We may assume without loss of generality $t \in \beta_1$.
		And we assume further that
		$t$ is also contained in a special cusp of $\alpha(x_1,x_2)$.
		Let $t_3 \in \beta_3$ for which $(t, t_3)$ is a special cusp of $\alpha(x_1,x_3)$,
		and let $t_2 \in \beta_2$ for which $(t, t_2)$ is a special cusp of $\alpha(x_1,x_2)$.
		We first exclude the case that the two cusps $(t, t_2), (t, t_3)$ are contained in two different cataclysms.
		As noted in Remark \ref{remark on upward cataclysm},
		if $t$ is contained in two different cataclysms $\mu_1, \mu_2$,
		then one of $\mu_1, \mu_2$ is an upward cataclysm and the other one is a downward cataclysm,
		and moreover,
		$t$ separates any $\{a,b\}$ with $a \in \mu_1 - \{t\}, b \in \mu_2 - \{t\}$.
		As $t \notin \beta_2, \beta_3$ and $\alpha(x_2, x_3) = \beta_3 \circ \overline{\beta_2}$,
		$t \notin \alpha(x_2, x_3)$.
		This implies that $t$ doesn't separate $\{x_2, x_3\}$ and thus
		doesn't separate $\{t_2, t_3\}$.
		It follows that $t, t_2, t_3$ are contained in the same cataclysm of $L$ (see Figure \ref{cusp} (d)).
		Let $\mu$ denote this cataclysm.
		Note that $\mu$ contains three cusps $(t, t_2), (t, t_3), (t_2, t_3)$ in
		$\alpha(x_1, x_2), \alpha(x_1, x_3), \alpha(x_2, x_3)$ respectively.
		We consider the following two possibilities:
		
		$\bullet$
		Assume
		$t \stackrel{_\mu}{<} t_3$.
		Then $\alpha(x_1,x_3)$ has a positive cusp at $\mu$.
		And we have either
		$t \stackrel{_\mu}{<} t_2$ or $t_2 \stackrel{_\mu}{<} t_3$.
		So at least one of the two cusps of $\alpha(x_1,x_2)$, 
		$\alpha(x_2,x_3)$ at $\mu$ is positive.
		It follows that
		$|n(x_1,x_3)- (n(x_1,x_2)+n(x_2,x_3))| = 1$.
		
		$\bullet$
		Assume
		$t \stackrel{_\mu}{>} t_3$.
		Then $\alpha(x_1,x_3)$ has a negative cusp at $\mu$,
		and either
		$t \stackrel{_\mu}{>} t_2$ or $t_2 \stackrel{_\mu}{>} t_3$.
		So at least one of the two cusps of $\alpha(x_1,x_2)$, 
		$\alpha(x_2,x_3)$ at $\mu$ is negative,
		and therefore
		$|n(x_1,x_3)- (n(x_1,x_2)+n(x_2,x_3))| = 1$.
	\end{case}
	
	Thus we have
	$n(x_1,x_3) \in \{n(x_1,x_2)+n(x_2,x_3)-1,n(x_1,x_2)+n(x_2,x_3)+1\}$.
\end{proof}

It follows that

\begin{cor}\label{linear}
	(a)
	If $n(x_1,x_2), n(x_2,x_3) > 0$,
	then
	$n(x_1,x_3) > 0$.
	
	(b)
	If $n(x_1,x_2), n(x_2,x_3) < 0$,
	then
	$n(x_1,x_3) < 0$.
\end{cor}

Let $\stackrel{_n}{<}$ be the order on $End_+(L)$ defined by
$x_1 \stackrel{_n}{<} x_2$ if $n(x_1,x_2) < 0$ for all $x_1,x_2 \in End_+(L)$.
By Lemma \ref{anti-sym}, Lemma \ref{nonzero} and Corollary \ref{linear},
we have

\begin{cor}\label{order}
	$\stackrel{_n}{<}$ is a linear order on $End_+(L)$.
\end{cor}

\begin{lm}\label{preserve}
	The action of $G$ on $End_+(L)$ preserves the order $\stackrel{_n}{<}$ on $End_+(L)$.
\end{lm}
\begin{proof}
	Recall from Assumption \ref{has orderable cataclysm},
	for any cataclysm $\mu$ of $L$ and $u,v \in \mu$ with 
	$u \stackrel{_\mu}{<} v$,
	we have $g(u) \stackrel{_{g(\mu)}}{<} g(v)$ for all $g \in G$.
	Hence the positivity and negativity for cusps are
	$\pi_1$-equivariant,
	and therefore $n(x_1,x_2)$ is $\pi_1$-equivariantly defined on $End_+(L)$.
	So $\stackrel{_n}{<}$ is preserved by the action of $G$.
\end{proof}

\begin{lm}\label{nontrivial}
	$G$ acts on $End_+(L)$ nontrivially.
\end{lm}
\begin{proof}
	Recall that $\mathcal{F}$ has two-sided branching
	and $L$ has infinitely many positive ends.
	We choose $x_1, x_2 \in End_+(L)$ such that
	$|n(x_1,x_2)| = 1$.
	Then $\alpha(x_1,x_2)$ has exactly one cusp.
	We denote this cusp by $(u,v)$.
	
	We choose $g \in G$ such that
	$g(u) \ne u$.
	If $g(x_1) = x_1$ and $g(x_2) = x_2$,
	then $g$ takes $\alpha(x_1,x_2)$ to itself,
	and thus $g$ fixes both $u$ and $v$.
	This contradicts to $g(u) \ne u$.
	So either $g(x_1) \ne x_1$ or $g(x_2) \ne x_2$.
	Therefore,
	$G$ acts on $End_+(L)$ nontrivially.
\end{proof}

Combining Lemma \ref{preserve} and Lemma \ref{nontrivial},
$G$ has a nontrivial quotient acting on $End_+(L)$ effectively that
preserves the order $\stackrel{_n}{<}$ on $End_+(L)$.
Thus,
$G$ has a nontrivial left orderable quotient.
By \hyperref[BRW]{[BRW, Theorem 3.2]},
$G$ is left orderable.

Similar to the order $\stackrel{_n}{<}$ on $End_+(L)$,
we can assign a linear order $\stackrel{_m}{<}$ to $End_-(L)$ which is invariant under the action of $G$.
For $x,y \in End(L)$,
define $x \prec y$ if one of the following conditions holds:

(1)
$x \in End_-(L)$, $y \in End_+(L)$. 

(2)
$x, y \in End_+(L)$ and 
$x \stackrel{_n}{<} y$.

(3)
$x, y \in End_-(L)$ and 
$x \stackrel{_m}{<} y$.

It's not hard to see that $\prec$ is a linear order on $End(L) = End_-(L) \cup End_+(L)$ and
$\prec$ is invariant under the action of $G$.

We have completed the proof of Theorem \ref{orderable cataclysm}.
In fact,
we can construct a left-invariant order of $G$.
We describe this construction in the next subsection.

\subsection{The construction of a left-invariant order of $G$}\label{subsection 3.3}

Let $$H = \{g \in G \mid g(y) = y \text{ for all } y \in End_+(L)\}.$$
Note that $H$ is a normal subgroup of $G$,
and $G / H$ acts on $End_+(L)$ effectively.
Let $N = G / H$.
By Lemma \ref{preserve},
the action of $N$ on $End_+(L)$ preserves the linear order $\stackrel{_n}{<}$ on $End_+(L)$.
Recall that if 
a group acts effectively on a linearly ordered set preserving the linear order,
then we can obtain a left-invariant order of this group
(see, for example, \hyperref[BRW]{[BRW, Theorem 2.4, proof]}).
So we can obtain a left-invariant order $\stackrel{_N}{<}$ of $N$ from 
the action of $N$ on $End_+(L)$.

If $H = \{1\}$,
then $G = N$ and $\stackrel{_N}{<}$ is a left-invariant order of $G$.
Now assume $H \ne \{1\}$.
We choose $x_1, x_2 \in End_+(L)$ with $|n(x_1,x_2)| = 1$.
Since $H$ fixes both of $x_1, x_2$,
every element of $H$ takes $\alpha(x_1,x_2)$ to itself, 
and thus $H$ fixes the two points in the unique cusp of $\alpha(x_1,x_2)$.
We choose a point $\lambda$ in this cusp.
Then $H \subseteq Stab_G(\lambda)$.
Note that $Stab_G(\lambda)$ is the fundamental group of an orientable surface.
We can choose a left-invariant order of $Stab_G(\lambda)$
(see, for example, \hyperref[BRW]{[BRW, Section 7]}),
which induces a left-invariant order $\stackrel{_H}{<}$ of $H$.
As shown in \hyperref[BRW]{[BRW, Lemma 2.3]},
the left-invariant order $\stackrel{_H}{<}$ of $H$ and
the left-invariant order $\stackrel{_N}{<}$ of $N = G / H$ define
a left-invariant order of $G$.

\subsection{The proof of Proposition \ref{Dehn filling}}\label{subsection 3.4}
We prove Proposition \ref{Dehn filling} in this subsection.

Suppose that $M$ admits a pseudo-Anosov flow $X$,
and we may assume that $X$ is not Anosov.
Recall that the \emph{stable foliation} (denoted $\mathcal{F}^{s}$) and
unstable foliation (denoted $\mathcal{F}^{u}$) of $X$ are
a pair of singular foliations whose
leaves are foliated by the flowlines of $X$ with the following properties:
the leaves of $\mathcal{F}^{s}, \mathcal{F}^{u}$ intersect transversely, 
and their intersections are exactly the flowlines of $X$.
A closed orbit of $X$ (i.e. an orbit homeomorphic to $S^{1}$) is called a \emph{singular orbit} if
it is exactly the singular set of a singular leaf of $\mathcal{F}^{s}$ or $\mathcal{F}^{u}$.
Note that $X$ contains finitely many singular orbits,
and each singular leaf of $\mathcal{F}^{s}$ or $\mathcal{F}^{u}$ contains exactly one singular orbit.

Let $\{\gamma_1,\ldots,\gamma_n\}$ denote the set of singular orbits of $X$ and
$\gamma = \bigcup_{i=1}^{n}\gamma_i$.
Let $N(\gamma_i)$ be a regular neighborhood of $\gamma_i$ and $T_i = \partial N(\gamma_i)$.
We denote by $N(\gamma) = \bigcup_{i=1}^{n}N(\gamma_i)$.
We adopt the following convention for the multislopes on $\bigcup_{i=1}^{n} T_i$.

\begin{convention}\rm\label{convention}
	(a)
	We denote by $l_i$ the singular leaf of $\mathcal{F}^{s}$ that
	contains $\gamma_i$.
	Then $l_i \cap \partial T_i$ is a collection of parallel curves in $N(\gamma_i)$.
	We call the slope of them the \emph{preferred framing} on $T_i$.
	
	(b)
	For any two slopes $\alpha, \beta$ on $T_i$,
	we denote by $\Delta(\alpha, \beta)$ the minimal geometric intersection number of $\alpha, \beta$.
	
	(c)
	Let $M_\gamma = M - Int(N(\gamma))$.
	For every multislope 
	$(r_1,\ldots,r_n)$ on $\bigcup_{i=1}^{n} T_i$,
	we denote by
	$M_\gamma(r_1,\ldots,r_n)$ the Dehn filling of $M - Int(N(\gamma))$ along $\partial N(\gamma)$ with the multislope $(r_1,\ldots,r_n)$.
\end{convention}

We complete the proof of Proposition \ref{Dehn filling} in the remainder of this subsection:

\begin{Dehn filling}\label{filling}
	Suppose that $\mathcal{F}^{s}$ is co-orientable.
	Let $N = M_\gamma(r_1,\ldots,r_n)$.
	Let $s_i$ denote the preferred framing on $T_i$.
	If $\Delta(r_i, s_i)= 1$ for every $1 \leqslant i \leqslant n$,
	then $N$ admits a co-orientable taut foliation with orderable cataclysm.
\end{Dehn filling}

We can split open $\mathcal{F}^{s}$ along the singular leaves to obtain
an essential lamination $\mathcal{L}^{s}$ of $M$.
Then $\mathcal{L}^{s}$ has 
$n$ complementary regions which contain $\gamma_1,\ldots,\gamma_n$ respectively,
and each one of them is a (finite sided) ideal polygon bundle over $S^{1}$.
We may assume every $N(\gamma_i)$ is contained in a complementary region of $\mathcal{L}^{s}$.
Then $\mathcal{L}^{s} \subseteq M - Int(N(\gamma))$,
and thus, 
$\mathcal{L}^{s}$ can be canonically identified with a lamination of
$N = M_\gamma(r_1,\ldots,r_n)$
(which we still denote by $\mathcal{L}^{s}$).
Since $\mathcal{F}^{s}$ is a co-orientable singular foliation of $M$,
$\mathcal{L}^{s}$ is a co-orientable lamination in $M$ and thus
is also a co-orientable lamination in $N$.
Henceforth $\mathcal{L}^{s}$ will always denote the lamination in $N$.

Note that $r_i \ne s_i$ for every $1 \leqslant i \leqslant n$
since $\Delta(r_i, s_i)= 1$.
So the complementary regions of $\mathcal{L}^{s}$ in $N$ are ideal polygon bundles over $S^{1}$.
Because $\mathcal{L}^{s}$ is co-orientable,
all ideal polygon fibers of these bundles have an even number of sides.
Recall that if a complementary region of a lamination in a $3$-manifold is
a bundle of ideal polygons with an even number of sides over $S^{1}$,
then we can fill it with monkey saddles
(see, for example, \hyperref[C]{[C, Example 4.19, Example 4.20]}).
Now we fill every complementary region of $\mathcal{L}^{s}$ in $N$ with
monkey saddles to obtain a co-orientable foliation $\mathcal{F}$ of $N$.
$\mathcal{F}$ is taut since it contains no compact leaf.
For more detail in this process,
we refer to
\hyperref[G2]{[G2]},
or see also \hyperref[C]{[C, Example 4.22]}).

We denote by $R_1,\ldots,R_n$ the complementary regions of $\mathcal{L}^{s}$ in $N$ for which
$T_i \subseteq R_i$.
Let $q: \widetilde{N} \to N$ be the universal covering space of $N$.
Let $\widetilde{\mathcal{L}^{s}}$ be the pull-back lamination of $\mathcal{L}^{s}$ in $\widetilde{N}$ and
let $\widetilde{\mathcal{F}}$ be the pull-back foliation of $\mathcal{F}$ in $\widetilde{N}$.
We denote by $Aut(\widetilde{N})$ the deck transformation group of $\widetilde{N}$.
We fix a co-orientation on $\mathcal{F}$ and an induced co-orientation on $\widetilde{\mathcal{F}}$.

It remains to show that $\mathcal{F}$ has orderable cataclysm:

\begin{lm}\label{orderable}
	Let $\mu$ a cataclysm of $\widetilde{\mathcal{F}}$.
	Then there is a linear order of the leaves in $\mu$ preserved by 
	the stabilizer subgroup of $Aut(\widetilde{N})$ with respect to $\mu$.
\end{lm}

To prove Lemma \ref{orderable},
we consider two different cases for $\mu$.

\begin{case2}\rm\label{case2 1}
	Assume that there is a component $\widetilde{R_i}$ of $q^{-1}(R_i)$ for some $1 \leqslant i \leqslant n$ such that
	every leaf in $\mu$ is a boundary leaf of $\widetilde{\mathcal{L}^{s}}$ contained in $\widetilde{R_i}$.
\end{case2}

Let $$\mu^{+} = \{\text{boundary leaves of } \widetilde{\mathcal{L}^{s}}
\text{ in } \partial \widetilde{R_i} \text{ with transverse orientation pointing out of } \widetilde{R_i}\},$$
$$\mu^{-} = \{\text{boundary leaves of } \widetilde{\mathcal{L}^{s}}
\text{ in } \partial \widetilde{R_i} \text{ with transverse orientation pointing in } \widetilde{R_i}\}.$$
We can observe that
each of $\mu^{+}, \mu^{-}$ is a cataclysm of $\widetilde{\mathcal{F}}$.
So $\mu$ is exactly either $\mu^{+}$ or $\mu^{-}$.

Let $H = \{g \in Aut(\widetilde{N}) \mid g(\widetilde{R_i}) = \widetilde{R_i}\}$.
Then $H$ is a cyclic group,
and $H$ is also the stabilizer subgroup of $Aut(\widetilde{N})$ with respect to 
both $\mu^{+}$ and $\mu^{-}$.
We denote by $\widetilde{T_i}$ the component of $q^{-1}(T_i)$ contained in $\widetilde{R_i}$.
Note that $\widetilde{T_i} \cong S^{1} \times \mathbb{R}$ and 
the stabilizer subgroup of $Aut(\widetilde{N})$ with respect to $\widetilde{T_i}$ is also $H$.
It's not hard to verify that

\begin{fact}\rm\label{permutation}
	Let $\alpha_1, \ldots, \alpha_k$ be a collection of disjoint essential simple closed curves on $T_i$ that
	have the same slope $\alpha$.
	Let $A =q^{-1}(\bigcup_{j=1}^{k}\alpha_j) \cap \widetilde{T_i}$. 
	Then $H$ acts on $\{\text{the components of } A\}$ by permutation.
	For every $x \in H$,
	the $\Delta(r_i,\alpha)$th-power of $x$ fixes every component of $A$.
\end{fact}

Recall from Convention \ref{convention},
we denote by $l_i$ the singular leaf of $\mathcal{F}^{s}$ containing $\gamma_i$, 
and $l_i \cap T_i$ is a collection of curves on $T_i$ with slope $s_i$.
Let
$\widetilde{l_i}$ denote the singular leaf of $\widetilde{\mathcal{F}^{s}}$ containing $\widetilde{\gamma_i}$.
Since $\Delta(r_i, s_i)= 1$,
Fact \ref{permutation} implies that
the action of $H$ on $\{\text{components of } \widetilde{l_i} \cap \widetilde{T_i}\}$ is identity.
It follows that every element of $H$ takes every leaf in $\mu^{-} \cup \mu^{+}$ to itself.
So Lemma \ref{orderable} holds in Case \ref{case2 1}.

\begin{case2}\rm\label{case2 2}
	Assume that there is no complementary region $\widetilde{R}$ of 
	$\widetilde{\mathcal{L}^{s}}$ in $\widetilde{N}$ such that
	all leaves in $\mu$ are the boundary leaves of $\widetilde{\mathcal{L}^{s}}$ in $\widetilde{R}$.
\end{case2}

The proof of Lemma \ref{orderable} for Case \ref{case2 2} basically follows from 
\hyperref[F3]{[F3]} or \hyperref[CD]{[CD, 3.4, Example 3.6]}.
In \hyperref[F3]{[F3]},
an explanation is given in terms of
cataclysms for singular foliations.
In \hyperref[CD]{[CD, 3.4, Example 3.6]},
an explanation is given in terms of 
cataclysms for essential laminations.
Now we give the proof in detail as follows.

Let $S_i$ denote the solid torus in $N$ bounded by $T_i$,
and let $\gamma^{'}_{i}$ be a core curve of $S_i$.
Let $\gamma^{'} = \bigcup_{i=1}^{n} \gamma^{'}_{i}$.
Then we can identify $N - \gamma^{'}$ with $M - \gamma$.
There is a flow $Y$ in $N$ induced from the flow $X$ in $M$ such that
$\gamma^{'}_{1}, \ldots, \gamma^{'}_{n}$ are flowlines of $Y$ and
$$Y \mid_{N - \gamma^{'}}  = X \mid_{M - \gamma}.$$
And there is a pair of singular foliations $\mathcal{E}^{s}, \mathcal{E}^{u}$ of $N$ induced from
$\mathcal{F}^{s}, \mathcal{F}^{u}$ such that
$$\mathcal{E}^{s} \mid_{N - \gamma^{'}}  = \mathcal{F}^{s} \mid_{M - \gamma},$$
$$\mathcal{E}^{u} \mid_{N - \gamma^{'}}  = \mathcal{F}^{u} \mid_{M - \gamma}.$$
Similar to $\mathcal{F}^{s}$ and $\mathcal{F}^{u}$,
the leaves of $\mathcal{E}^{s}$ and $\mathcal{E}^{u}$ intersect transversely,
and their intersections are exactly the flowlines of $Y$.
We denote by $\widetilde{Y}$ the pull-back flow of $Y$ in $\widetilde{N}$ and
denote by $\widetilde{\mathcal{E}^{s}}, \widetilde{\mathcal{E}^{u}}$
the pull-back singular foliations of $\mathcal{E}^{s}, \mathcal{E}^{u}$ in
$\widetilde{N}$.
Note that $\mathcal{L}^{s}$ can be obtained from splitting open $\mathcal{E}^{s}$ along
the singular leaves of $\mathcal{E}^{s}$.

\begin{figure}\label{approximate}
	\includegraphics[width=0.4\textwidth]{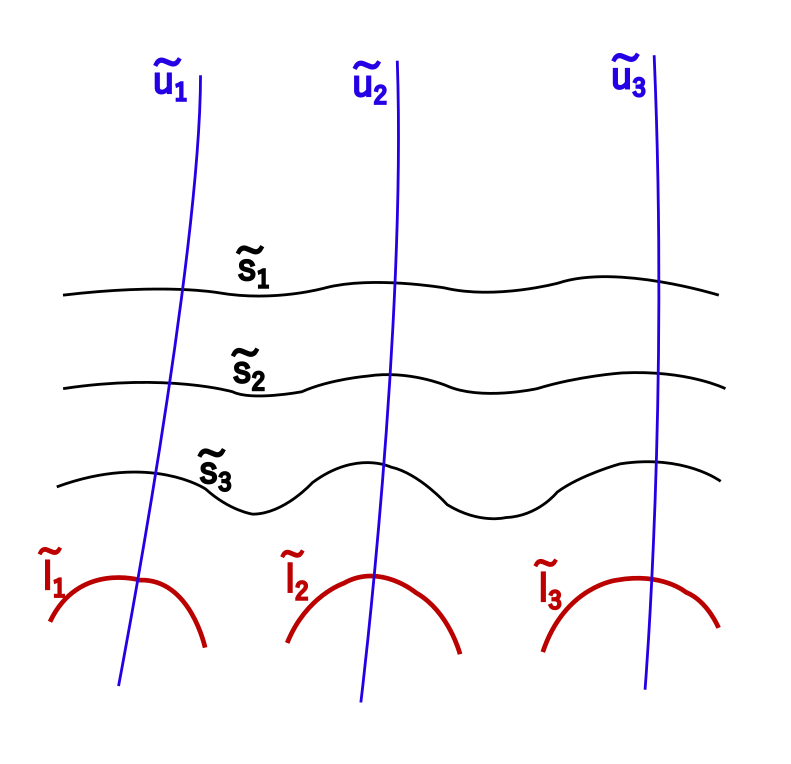}
	\caption{Let $\widetilde{l_1}, \widetilde{l_2}, \widetilde{l_3}$ be
	three leaves in $\mu$.
    We can choose three non-singular leaves $\widetilde{u_1}, \widetilde{u_2}, \widetilde{u_3}$ of 
$\widetilde{\mathcal{E}^{u}}$ with $\widetilde{u_i} \cap \widetilde{l_i} \ne \emptyset$ ($i = 1,2,3$).
    For each leaf of $\widetilde{\mathcal{E}^{s}}$ that intersets all of $\widetilde{u_1}, \widetilde{u_2}, \widetilde{u_3}$,
there is a natural linear order between $\widetilde{u_1}, \widetilde{u_2}, \widetilde{u_3}$ given from this leaf.
In the picture,
we draw the projections of $\widetilde{l_i}, \widetilde{u_j}$ ($i, j \in \{1,2,3\}$) into the orbit space
$\widetilde{N} / \widetilde{Y}$,
and we draw the projections of three non-singular leaves $\widetilde{s_1}, \widetilde{s_2}, \widetilde{s_3}$ of
$\widetilde{\mathcal{E}^{s}}$ that intersect each of $\widetilde{u_1}, \widetilde{u_2}, \widetilde{u_3}$.
The linear orders between $\widetilde{u_1}, \widetilde{u_2}, \widetilde{u_3}$ given from
$\widetilde{s_1}, \widetilde{s_2}, \widetilde{s_3}$ are equal.}
\end{figure}

For a non-singular leaf $\widetilde{s}$ of $\widetilde{\mathcal{E}^{s}}$,
the set of non-singular leaves of $\widetilde{\mathcal{E}^{u}}$ intersecting $\widetilde{s}$ naturally possesses a linear order.
As choosing $\widetilde{s}$ to be sufficiently approximate to the leaves in $\mu$,
this linear order on $\widetilde{s}$ extends to a $\pi_1$-equivariant linear order on $\mu$.
We explain this further below.

\begin{lm}\label{approaching}
	There is a sequence of leaves in $\widetilde{\mathcal{L}^{s}}$ that
	converges to all leaves in $\mu$ simultaneously,
	and each of them is not a boundary leaf of $\widetilde{\mathcal{L}^{s}}$.
\end{lm}
\begin{proof}
	Let $\widetilde{\lambda}$ be a leaf of $\widetilde{\mathcal{L}^{s}}$.
	Then there is a leaf $\widetilde{\lambda_*}$ of $\widetilde{\mathcal{E}^{s}}$ such that
	(1)
	if $\widetilde{\lambda}$ is not a boundary leaf of $\widetilde{\mathcal{L}^{s}}$,
	then $\widetilde{\lambda_*}$ is a non-singular leaf of $\widetilde{\mathcal{E}^{s}}$ that
	can be canonically identified with $\widetilde{\lambda}$,
	(2)
	if $\widetilde{\lambda}$ is a boundary leaf of $\widetilde{\mathcal{L}^{s}}$,
	then $\widetilde{\lambda_*}$ is a singular leaf of $\widetilde{\mathcal{E}^{s}}$ such that
	$\widetilde{\lambda}$ is obtained from splitting open along $\widetilde{\lambda_*}$.
	We call $\widetilde{\lambda_*}$ the
	\emph{original leaf} of $\widetilde{\lambda}$.
	The assumption of Case \ref{case2 2} guarantees that
	the original leaves of all leaves in $\mu$ are distinct from each other.
	
	We denote by $\mu_* = \{\text{original leaves of leaves in } \mu\}$.
	Then there are leaves $\{l(t) \mid t \in (0,1]\}$ of $\widetilde{\mathcal{E}^{s}}$ such that
	for any point $x$ of some leaf in $\mu_*$,
	there is a collection of points $x_t \in l(t)$ ($t \in (0,1]$) with
	$$\lim_{t \to 0} x_t = x.$$
	Note that $\widetilde{\mathcal{E}^{s}}$ contains only countably many singular leaves.
	So there is a sequence $\{t_i\}_{i \in \mathbb{N}}$ in $(0,1]$ converging to $0$ such that
	every $l(t_i)$ is not a singular leaf.
	Thus $\{l(t_i)\}_{i \in \mathbb{N}}$ is a collection of
	non-singular leaves of $\widetilde{\mathcal{E}^{s}}$ that converges to all leaves in $\mu_*$ simultaneously.
	$\{l(t_i)\}_{i \in \mathbb{N}}$ can be canonically identified with
	a sequence of leaves of $\widetilde{\mathcal{L}^{s}}$ that
	converges to all leaves in $\mu$ simultaneously,
	and each of them is not a boundary leaf of $\widetilde{\mathcal{L}^{s}}$.
\end{proof}

For a non-singular leaf $\widetilde{\lambda}$ of $\widetilde{\mathcal{E}^{s}}$,
we denote by $\Psi(\widetilde{\lambda})$ the $1$-dimensional foliation of $\widetilde{\lambda}$
consisting of the flowlines of $\widetilde{Y}$ and
denote by $L(\Psi(\widetilde{\lambda}))$ the leaf space of $\Psi(\widetilde{\lambda})$.
Then $L(\Psi(\widetilde{\lambda}))$ is homeomorphic to $\mathbb{R}$.
Given an orientation on $\widetilde{\lambda}$,
we can assign $\Psi(\widetilde{\lambda})$ a transverse orientation induced from
the orientation on $\widetilde{\lambda}$ and the orientations on the leaves of $\Psi(\widetilde{\lambda})$
(as the orientations on the flowlines).
We will henceforth assume $\Psi(\widetilde{\lambda})$ has this transverse orientation if
$\widetilde{\lambda}$ is oriented.

Note that $\mathcal{E}^{s}$ has a co-orientation induced from $\mathcal{F}^{s}$.
This induces a co-orientation on $\widetilde{\mathcal{E}^{s}}$
and also induces an orientation on 
every non-singular leaf of $\widetilde{\mathcal{E}^{s}}$,
where the orientations on these non-singular leaves are consistent with the co-orientation on 
$\widetilde{\mathcal{E}^{s}}$.
Since $\mathcal{E}^{s}$ is co-orientable,
the deck transformations of $\widetilde{N}$ preserve the co-orientation on $\widetilde{\mathcal{E}^{s}}$ and thus
preserve the orientation on every non-singular leaf of $\widetilde{\mathcal{E}^{s}}$.
For two distinct leaves $\widetilde{u_1}, \widetilde{u_2}$ of $\widetilde{\mathcal{E}^{u}}$ that
both intersect a leaf $\widetilde{\lambda}$ of $\widetilde{\mathcal{E}^{s}}$,
we denote by $l_1 = \widetilde{u_1} \cap \widetilde{\lambda}$,
$l_2 = \widetilde{u_2} \cap \widetilde{\lambda}$,
and we say $\widetilde{u_1} < \widetilde{u_2}$ at $\widetilde{\lambda}$ 
(resp. $\widetilde{u_1} > \widetilde{u_2}$ at $\widetilde{\lambda}$) if 
there is a positively oriented transversal (resp. negatively oriented transversal) of $\Psi(\widetilde{\lambda})$ 
from $l_1$ to $l_2$.
It's not hard to observe that if $\widetilde{u_1} < \widetilde{u_2}$ at $\widetilde{\lambda}$,
then $\widetilde{u_1} < \widetilde{u_2}$ at $\widetilde{\lambda_0}$ for 
any non-singular leaf $\widetilde{\lambda_0}$ of $\widetilde{\mathcal{E}^{s}}$ that
intersects both of $\widetilde{u_1}, \widetilde{u_2}$.
Moreover,
because the deck transformations of $\widetilde{N}$ preserve
the orientation on $\widetilde{\lambda}$ and the orientations on the flowlines of $\widetilde{Y}$,
for any $g \in Aut(\widetilde{N})$, 
we have
$g(\widetilde{u_1}) < g(\widetilde{u_2})$ at $g(\widetilde{\lambda})$ if and only if
$\widetilde{u_1} < \widetilde{u_2}$ at $\widetilde{\lambda}$.

We can split open $\mathcal{E}^{u}$ along the singular leaves of $\mathcal{E}^{u}$ to obtain
an essential lamination $\mathcal{L}^{u}$ of $N$.
Let $\widetilde{\mathcal{L}^{u}}$ be the pull-back lamination of $\mathcal{L}^{u}$ in $\widetilde{N}$.
For two leaves $\widetilde{u_1}, \widetilde{u_2}$ of $\widetilde{\mathcal{L}^{u}}$ that
both intersect a leaf $\widetilde{\lambda}$ of $\widetilde{\mathcal{L}^{s}}$,
if both of $\widetilde{u_1}, \widetilde{u_2}$ are not boundary leaves of $\widetilde{\mathcal{L}^{u}}$ and
$\widetilde{\lambda}$ is not a boundary leaf of $\widetilde{\mathcal{L}^{s}}$,
then they can be canonically identified with some non-singular leaves of
$\widetilde{\mathcal{E}^{u}}$ or $\widetilde{\mathcal{E}^{s}}$,
and thus we can define $\widetilde{u_1} < \widetilde{u_2}$ or
$\widetilde{u_1} > \widetilde{u_2}$ at $\widetilde{\lambda}$ according to the above discussion.

For arbitrary two leaves $\widetilde{\lambda_1}, \widetilde{\lambda_2}$ in $\mu$,
we can choose two non-boundary leaves 
$\widetilde{u_1}, \widetilde{u_2}$ of $\widetilde{\mathcal{L}^{u}}$ such that 
$\widetilde{u_1}$ intersects $\widetilde{\lambda_1}$,
$\widetilde{u_2}$ intersects $\widetilde{\lambda_2}$.
By Lemma \ref{approaching},
there exists a non-boundary leaf $\widetilde{l}$ of $\widetilde{\mathcal{L}^{s}}$ 
(sufficiently close to $\widetilde{\lambda_1}$ and $\widetilde{\lambda_2}$) such that $\widetilde{l}$
intersects both of $\widetilde{u_1}, \widetilde{u_2}$.
We define $\widetilde{\lambda_1} < \widetilde{\lambda_2}$ if
$\widetilde{u_1} < \widetilde{u_2}$ at $\widetilde{l}$ and
define $\widetilde{\lambda_1} > \widetilde{\lambda_2}$ if
$\widetilde{u_1} > \widetilde{u_2}$ at $\widetilde{l}$.
By the above discussion,
this definition is independent of the choice of $\widetilde{l}$,
and
$\widetilde{\lambda_1} < \widetilde{\lambda_2}$ implies
$h(\widetilde{\lambda_1}) < h(\widetilde{\lambda_2})$ for any $h \in Aut(\widetilde{N})$.
So Lemma \ref{orderable} holds in Case \ref{case2 2}.
This completes the proof of Proposition \ref{Dehn filling}.

\end{document}